\documentclass[leqno]{siamltex704}
\usepackage{amsmath}
\usepackage{graphicx}
\usepackage{mathrsfs}
\usepackage{bm}
\usepackage{float}
\usepackage{amsfonts,amssymb}
\usepackage{dsfont}
\usepackage{pifont}
\usepackage{tikz}
\usepackage{wrapfig} % Allows in-line images
\usepackage{hyperref}
\usepackage{multirow}
\usepackage{mathtools}
\usepackage{subfigure}
\usepackage{color}
\newtheorem{remark}{Remark}[section]
\def\S{{\mathfrak s}}
\def\T{{\mathcal T}}

\def\O{{\mathcal O}}
\def\B{{\mathcal B}}
\def\A{{\mathcal A}}

\def\bn{{\bf n}}

\def\pT{{\partial T}}
\def\3bar{{|\!|\!|}}
\def\bbeta{{\boldsymbol\beta}}
\newtheorem{algorithm}{SWG Algorithm}[section]
\newtheorem{FD-algorithm}{5-Point Finite Difference Algorithm}[section]

\newcommand{\vertiii}[1]{{\left\vert\kern-0.25ex\left\vert\kern-0.25ex\left\vert #1
    \right\vert\kern-0.25ex\right\vert\kern-0.25ex\right\vert}}
\title{A Simplified Weak Galerkin Finite Element Method: Algorithm and Error Estimates}

\author{Yujie Liu\thanks{School of Data and Computer Science, Sun Yat-sen University, Guangzhou, 510275, China (liuyujie5@mail.sysu.edu.cn). The research of Liu was partially supported by Guangdong Provincial Natural Science Foundation (No. 2017A030310285), Shandong Provincial natural Science Foundation (No. ZR2016AB15) and Youthful Teacher Foster Plan Of Sun Yat-Sen University (No. 171gpy118),} \and Junping Wang
\thanks{Division of Mathematical Sciences, National Science Foundation, Alexandria, VA 22314 (jwang@nsf.gov). The research of Wang was supported by the NSF IR/D program, while working at National Science Foundation. However, any opinion, finding, and conclusions or recommendations
expressed in this material are those of the author and do not
necessarily reflect the views of the National Science Foundation.}}
\begin{document}

\maketitle
%\linenumbers
\begin{abstract}
In this article a simplified weak Galerkin finite element method is developed for the Dirichlet boundary value problem of convection-diffusion-reaction equations. The simplified weak Galerkin method utilizes only the degrees of freedom on the boundary of each element and, hence, has significantly reduced computational complexity over the regular weak Galerkin finite element method. A stability and some optimal order error estimates in the $H^1$ and $L^2$ norms are established for the corresponding numerical solutions. Numerical results are presented to verify the theory error estimates and a superconvergence phenomena on rectangular partitions.
\end{abstract}



\pagestyle{myheadings}

%
%\noindent {\bf Mathematics Subject Classification (2010)} 65N15;
%65N30; 41A30

\section{Introduction}
This paper is concerned with the development of a simplified formulation for the weak Galerkin finite element method for second order elliptic equations. For simplicity, consider the model problem that seeks an unknown function $u=u(x)$ satisfying
\begin{eqnarray}
-\nabla\cdot(\alpha\nabla u) + \bbeta\cdot\nabla u + cu&=&f\quad {\rm in}\  \Omega  \label{ellipticbdy}\\
u&=&g\quad {\rm on}\ \partial\Omega \label{ellipticbc}
\end{eqnarray}
where $\Omega$ is a bounded polytopal domain in $\mathbb{R}^d \;(d\ge 2)$ with boundary $\partial\Omega$, $\alpha=\alpha(x)$ is the diffusion coefficient, $\bbeta=\bbeta(x)$ is the convection, and $c=c(x)$ is the reaction coefficient in relevant applications. We assume that $\alpha$ is sufficient smooth, $\bbeta\in [W^{1,\infty}(\Omega)]^d$, and $c$ is piecewise smooth with respect to a partition of the domain. For well-posedness of the problem \eqref{ellipticbdy}-\eqref{ellipticbc}, we assume $f=f(x)\in L^2(\Omega)$, $g=g(x)\in H^{\frac12}(\partial\Omega)$, and
\begin{equation}\label{EQ:positive}
c-\frac12\nabla\cdot\bbeta \ge 0,\qquad \alpha(x) \ge \alpha_0 \qquad \forall x \in \Omega
\end{equation}
for a constant $\alpha_0>0$.

The model problem \eqref{ellipticbdy}-\eqref{ellipticbc} arises from many scientific applications such as fluid flow in porous media. Mostly importantly, this model problem has served, and still serves, the scientific computing community as a testbed in the search and design of new and efficient computational algorithms for partial differential equations. The classical Galerkin finite element method (see, e.g., \cite{ciarlet-fem,StrangFix,gr}) is particularly a numerical technique originated from the study of elliptic problems closed related to \eqref{ellipticbdy}-\eqref{ellipticbc} or its variations. In the last three decades, various finite element methods using discontinuous trial and test functions, including discontinuous Galerkin (DG) methods and weak Galerkin (WG) methods, have been developed for numerical solutions of partial differential equations. These developments were often tested over testbed problems such as \eqref{ellipticbdy}-\eqref{ellipticbc} before they were generalized or applied to more complex problems in science and engineering. The DG method, also known as the interior penalty method in different contexts, was originated in early 70s of the last century for a numerical study of model problems such as \eqref{ellipticbdy}-\eqref{ellipticbc}; see, e.g., \cite{babuska-dg, douglas-dg, nitsche-dg, wheeler-dg} for early incubations and \cite{arnold-dg,pietro-ern,hesthaven, riviere} for a detailed discussion and recent developments.

The weak Galerkin finite element method is a recently developed discretization framework for partial differential equations \cite{WangYe_2013,wy3655,mwy,ww-survey}. With new concepts referred to as weak differential operators (e.g., weak gradient, weak curl, weak Laplacian etc.) and weak continuity through the use of various stabilizers, the method allows the use of totally discontinuous functions and provides stable numerical schemes that are parameter-independent or free of locking \cite{LockingFree}. For the convection-diffusion-reaction equation \eqref{ellipticbdy}-\eqref{ellipticbc}, the recent work in the context of weak Galerkin includes the algorithm developed and analyzed in \cite{Chen-CD-WG}, the one in \cite{LinRunchang} for singularly perturbed problems, and an earlier one in \cite{ZhangTie}. The WG finite element method has been rapidly developed and applied to several different types of problems, including second order elliptic problems, the Stokes and Navier-Stokes equations, the biharmonic and elasticity equations, div-curl systems and the Maxwell's equations, etc. The latest development of the WG methods is the prime-dual formulation for problems that are either nonsymmetric or do not have variational forms friendly for numerical use. Details on the new developments can be found in \cite{WangWang_2016} for second order elliptic equations in nondivergence form, \cite{ww2017} for the Fokker-Planck equation, and \cite{WangWang-EC} for elliptic Cauchy problems.

The typical WG method for the model problem \eqref{ellipticbdy}-\eqref{ellipticbc} seeks weak finite element approximations $u_h=\{u_0, u_b\}$ satisfying $u_b|_{\partial \Omega} = Q_b g$ and
\begin{equation}\label{EQ:826:001}
S(u_h, v)+ (\alpha \nabla_w u_h, \nabla_w v) + (\bbeta\cdot\nabla_w u_h, v_0) + (c u_0, v_0) = (f,v_0)
\end{equation}
for all test functions $v=\{v_0, v_b\}$ satisfying $v_b|_{\partial\Omega}=0$, where $Q_bg$ is an interpolation of the Dirichlet boundary data, $\nabla_w$ is the discrete weak gradient operator, and $S(\cdot,\cdot)$ is a properly selected stabilizer that gives weak continuities for the numerical solutions. The numerical solution $u_h$ consists of two components: the approximation $u_0$ on each element and the approximation $u_b$ on the boundary of each element. To reduce the computational complexity, some hybridized formulations have been introduced in \cite{MWY-HWG,ww-hwg} for the method when applied to the diffusion equation and the biharmonic equation through the elimination of the degrees of freedom associated with the unknown function $u_0$ locally on each element. In the superconvergence study for WG \cite{LiDanWW} on rectangular elements, this hybridized formulation was further simplified in the description of the numerical algorithm, yielding a {\em simplified weak Galerkin (SWG) finite element scheme} for the diffusion equation. In our further investigation of the SWG to the convection-diffusion-reaction equation \eqref{ellipticbdy}, we came to the conclusion that {\em SWG represents a new discretization scheme that is different from the usual WG} through a simple elimination of the unknown $u_0$. As a result, we believe that a systematic study of the SWG for the convection-diffusion-reaction problem \eqref{ellipticbdy}-\eqref{ellipticbc} should be conducted for its stability and convergence. This paper is in response to this observation and shall provide a mathematical theory for the stability and the convergence of the simplified weak Galerkin finite element method for the model problem \eqref{ellipticbdy}-\eqref{ellipticbc}. We believe that the result of this paper can be extended to other types of modeling equations.

The paper is organized as follows: In Section \ref{section-SWG-polymesh}, we shall describe the {\em simplified weak Galerkin} finite element method for \eqref{ellipticbdy}-\eqref{ellipticbc} on general polygonal partitions. In Section \ref{Section:ESM}, we shall present a computational formula for the element stiffness matrices and the element load vectors from SWG. In Section \ref{Section:SWP}, we provide a mathematical theory for the stability and well-posedness of the SWG scheme. Sections \ref{sectionEE} and \ref{sectionEE-L2} are devoted to a discussion of the error estimates in a discrete $H^1$ and the $L^2$ norm for the numerical solutions. Finally, in Section \ref{numerical-experiments}, we present some numerical results to demonstrate the efficiency and accuracy of the SWG method.

Throughout the rest of the paper, we assume $d=2$ and shall use the standard notations for Sobolev spaces and norms \cite{ciarlet-fem,gr}. For any open set $D\subset\mathbb{R}^{2}$, $\|\cdot\|_{s,D}$ and $(\cdot,\cdot)_{s,D}$ denote the norm and inner-product in the Sobolev space $H^s(D)$ consisting of square integrable partial derivatives up to order $s$. When $s=0$ or $D=\Omega$, we shall drop the corresponding subscripts in the norm and inner-product notation.

\section{Algorithm on Polymesh}\label{section-SWG-polymesh}
Assume that the domain is of polygonal type and is partitioned into non-overlap polygons $\T_h=\{T\}$ that are shape regular. For each $T\in \T_h$, denote by $h_T$ its diameter and by $N$ the number of edges. For each edge $e_i, \ i=1,\ldots, N$, denote by $M_i$ the midpoints and $\bn_i$ the outward normal direction of $e_i$ (see Fig. \ref{fig.hexahedron} for an illuatration). The meshsize of $\T_h$ is defined as $h=\max_{T\in\T_h} h_T$.

Let $v_b$ be a piecewise constant function defined on the boundary of $T$, i.e.,
\[
v_b|_{e_i} =v_{b,i},
\]
with $v_{b,i}$ being a constant. We define the weak gradient of $v_b$ on $T$ by:
\begin{equation}\label{DefWGpoly}
\nabla_w v_b:=\displaystyle\frac{1}{|T|}\sum_{i=1}^N v_{b,i}|e_i|\bf{n_i},
\end{equation}
where $|e_i|$ is the length of the edge $e_i$ and $|T|$ is the area of the element $T$. It is not hard to see that the weak gradient $\nabla_w v_b$ satisfies the following equation:
\begin{equation}\label{DefWGpoly-new}
(\nabla_w v_b, \bm{\phi})_T=\langle v_b, \bm{\phi}\cdot\bn\rangle_\pT
\end{equation}
for all constant vector $\bm{\phi}$. Here and in what follows of the paper, $\langle\cdot,\cdot\rangle_\pT$ stands for the usual inner product in $L^2(\pT)$.

Denote by $W(T)$ the space of piecewise constant functions on $\pT$. The global finite element space $W(\T_h)$ is constructed by patching together all the local elements $W(T)$ through single values on interior edges. The subspace of $W_h(\T_h)$ consisting of functions with vanishing boundary value is denoted as $W_h^0(\T_h)$.

We use the conventional notation of ${P}_j(T)$ for the space of polynomials of degree $j\ge 0$ on $T$. For each $v_b\in W(T)$, we associate it with a linear extension in $T$, denoted as $\S(v_b)\in {P}_1 (T)$, satisfying
\begin{equation}\label{Def.extension}
\sum_{i=1}^{N}(\S(v_b)(M_i) -v_{b,i})\phi(M_i)|e_i|=0,\quad \forall\; \phi\in {P}_1(T).
\end{equation}
It is easy to see that $\S(u_b)$ is well defined by \eqref{Def.extension}, and its computation is local and straightforward. In fact, $\S(u_b)$ can be viewed as an extension of $u_b$ from $\partial T$ to $T$ through a least-squares fitting.

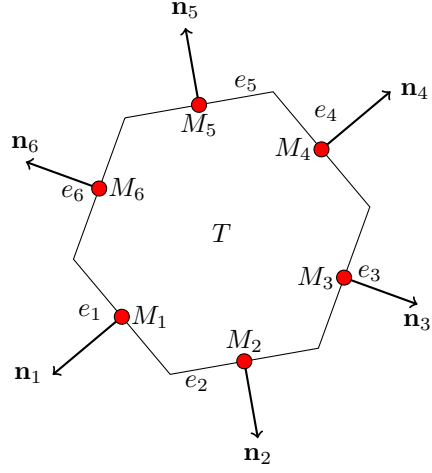
\begin{figure}[!h]
\begin{center}
\begin{tikzpicture}[rotate=40, scale =2.0]

    \path (-0.8660, 0.5) coordinate (A1);
    \path (-0.8660,-0.5) coordinate (A2);
    \path (0.     ,-1  ) coordinate (A3);
    \path (0.8660 ,-0.5) coordinate (A4);
    \path (0.8660 ,0.5 ) coordinate (A5);
    \path (0.     ,1.0 ) coordinate (A6);

    \path (0.0,0.0) coordinate (center);

    \path (-0.8660,0.   )    coordinate (A1half);
    \path (-0.433 ,-0.75)    coordinate (A2half);
    \path ( 0.433 ,-0.75)    coordinate (A3half);
    \path ( 0.8660,0.   )    coordinate (A4half);
    \path ( 0.433 , 0.75)    coordinate (A5half);
    \path (-0.433 , 0.75)    coordinate (A6half);

    \path (A1half) ++(-0.25 ,  0.25)  coordinate (A1halfe);
    \path (A2half) ++(-0.225, -0.0 )  coordinate (A2halfe);
    \path (A3half) ++( 0.225, -0.0 )  coordinate (A3halfe);
    \path (A4half) ++( 0.25 ,  0.25)  coordinate (A4halfe);
    \path (A5half) ++( 0.225,  0.0 )  coordinate (A5halfe);
    \path (A6half) ++(-0.225,  0.0 )  coordinate (A6halfe);

    \path (A1half) ++(-0.6   ,   0   )  coordinate (A1To);
    \path (A2half) ++(-0.2598,  -0.45)  coordinate (A2To);
    \path (A3half) ++( 0.2598,  -0.45)  coordinate (A3To);
    \path (A4half) ++( 0.6   ,   0   )  coordinate (A4To);
    \path (A5half) ++( 0.2598,   0.45)  coordinate (A5To);
    \path (A6half) ++(-0.2598,   0.45)  coordinate (A6To);
    \draw (A6) -- (A1) -- (A2) -- (A3)--(A4)-- (A5)--(A6);

    \draw node at (center) {$T$};
    \draw node[right] at (A1half) {$M_{1}$};
    \draw node[above] at (A2half) {$M_{2}$};
    \draw node[left] at (A3half) {$M_{3}$};
    \draw node[left]  at (A4half) {$M_{4}$};
    \draw node[below] at (A5half) {$M_{5}$};
    \draw node[right] at (A6half) {$M_{6}$};

    \draw node[right] at (A1halfe) {$e_{1}$};
    \draw node[left] at (A2halfe) {$e_{2}$};
    \draw node[below] at (A3halfe) {$e_{3}$};
    \draw node[below]  at (A4halfe) {$e_{4}$};
    \draw node[right] at (A5halfe) {$e_{5}$};
    \draw node[above] at (A6halfe) {$e_{6}$};

    \draw[->,thick] (A1half) -- (A1To) node[left] {$\mathbf{n}_1$};
    \draw[->,thick] (A2half) -- (A2To) node[below]{$\mathbf{n}_2$};
    \draw[->,thick] (A3half) -- (A3To) node[below]{$\mathbf{n}_3$};
    \draw[->,thick] (A4half) -- (A4To) node[right]{$\mathbf{n}_4$};
    \draw[->,thick] (A5half) -- (A5To) node[above]{$\mathbf{n}_5$};
    \draw[->,thick] (A6half) -- (A6To) node[above]{$\mathbf{n}_6$};

    \draw [fill=red] (A1half) circle (0.05cm);
    \draw [fill=red] (A2half) circle (0.05cm);
    \draw [fill=red] (A3half) circle (0.05cm);
    \draw [fill=red] (A4half) circle (0.05cm);
    \draw [fill=red] (A5half) circle (0.05cm);
    \draw [fill=red] (A6half) circle (0.05cm);

\end{tikzpicture}
\end{center}
\caption{An illustrative polygonal element.}
\label{fig.hexahedron}
\end{figure}

On each element $T \in \T_h $, we introduce the following bilinear forms:
\begin{eqnarray}
a_T(u_b,v_b)&:= & (\alpha\nabla_w u_b, \nabla_w v_b)_T, \label{EQ:aform}\\
b_T(u_b,v_b)&:= & (\bbeta\cdot\nabla_w u_b, \S(v_b))_T,\label{EQ:bform}\\
c_T(u_b,v_b)&:= & (c\S(u_b), \S(v_b))_T.\label{EQ:cform}
\end{eqnarray}
For simplicity, we set
\begin{equation}\label{EQ:Fulla}
\mathcal{B}_T(u_b,v_b):=a_T(u_b,v_b) + b_T(u_b,v_b) + c_T(u_b,v_b)
\end{equation}
for $u_b, v_b \in W(T)$.
We further introduce the stabilizer
\begin{equation}\label{EQ:stabilizer}
\begin{split}
S_T(u_b,v_b):= & h^{-1}\sum_{i=1}^N (\S(u_b)(M_i)-u_{b,i})(\S(v_b)(M_i)-v_{b,i})|e_i|\\
             = & h^{-1}\langle Q_b\S(u_b)-u_{b},Q_b\S(v_b)-v_{b}\rangle_{\partial T},
             \end{split}
\end{equation}
where $Q_b$ is the $L^2$ projection operator onto $W(T)$; namely $Q_b u$ is the average of $u$ on each edge. In particular, $Q_b(g)$ is well-defined and takes the average of the Dirichlet data on each boundary edge.
\medskip

\begin{algorithm}
The simplified weak Galerkin (SWG) scheme for the elliptic equation \eqref{ellipticbdy}-\eqref{ellipticbc} seeks
$u_b \in W_h(\T_h)$ satisfying $u_b = Q_b(g)$ on $\partial\Omega$ and
\begin{equation}\label{equation.SWG}
\A(u_b, v_b) =(f, \S(v_b))\qquad \forall v_b \in W_h^0(\T_h),
\end{equation}
where $\A(u_b,v_b):=\kappa S(u_b,v_b)+ \B(u_b, v_b)$,
\begin{eqnarray*}
  S(u_b,v_b) &=& \sum_{T\in\T_h} S_T(u_b,v_b),\\
  \B(u_b,v_b)&=& \sum_{T\in\T_h} \mathcal{B}_T(u_b, v_b)
\end{eqnarray*}
are bilinear forms in $W_h(\T_h)$ and $(f,\S(v_b)):=\sum_{T\in\T_h} (f, \S(v_b))_T$ is a linear form in $W_h(\T_h)$.
\end{algorithm}
\medskip

\section{Element Stiffness Matrices}\label{Section:ESM}
The simplified weak Galerkin finite element method \eqref{equation.SWG} is user-friendly in computer implementation. In this section, we present a formula for the computation of the element stiffness matrices and the element load vector on general polygonal elements.

\begin{theorem}
Let $T\in \T_h$ be a polygonal element of $N$ sides. Denote by $X_{u_b}$ the vector representation of $u_b$ given by $(u_{b,1},u_{b,2},\ldots,u_{b,N})^T$. Then, the element stiffness matrix and the element load vector for the SWG scheme \eqref{equation.SWG} are given in a block matrix form as follows:
\begin{equation}\label{EQ:ESM:01}
(\kappa h^{-1}A^T + B + R + C)X_{u_b} \cong F,
\end{equation}
where the block components in \eqref{EQ:ESM:01} are given by:
\begin{itemize}
\item[(1)] $A:=\{a_{i,j}\}_{i,j=1}^N = E- EM(M^TEM)^{-1}M^TE$,
\item[(2)] $B:=\{b_{i,j}\}_{i,j=1}^N$, with $b_{i,j} = (\alpha\bm{n}_i,\bm{n}_j)_{T}\displaystyle\frac{|e_i||e_j|}{|T|^2}$,
    \item[(3)] $R:=\{r_{ij}\}_{i,j=1}^N$, with $r_{ij}= \frac{|e_j|}{|T|}\int_T \bbeta\cdot\bm{n}_j \zeta_i dT$,
        \item[(4)] $C:=\{c_{ij}\}_{i,j=1}^N$, with $c_{ij}= \int_T c\zeta_j\zeta_i dT$,
\item[(5)] $F:=\{f_{i}\}_{i=1}^N$, with $f_{i} = \int_{T}f(x,y) \zeta_i(x,y) dT$,
 \item[(6)] $D:=\{d_{j,i}\}_{3\times N}=(M^TEM)^{-1}M^TE$ and $\zeta_i = d_{1,i} + d_{2,i} (x-x_T) + d_{3,i} (y-y_T)$,
\item[(7)] $M$ and $E$ are given by
\begin{equation*}
M=
\begin{bmatrix}
1      & x_{1} - x_T & y_{1} - y_T\\
1      & x_{2} - x_T & y_{2} - y_T\\
\vdots & \vdots        & \vdots       \\
1      & x_{N} - x_T & y_{N} - y_T\\
\end{bmatrix}_{N\times3},\;
E=
\begin{bmatrix}
|e_1| &       &           &       \\
      & |e_2| &           &       \\
      &       & \ddots    &       \\
      &       &           & |e_N| \\
\end{bmatrix}_{N\times N}.
\end{equation*}
\end{itemize}
Here $M_T=(x_T,y_T)$ is any point on the plane (e.g., the center of $T$ as a specific case), $(x_i, y_i)$ is the midpoint of $e_i$, $|e_i|$ is the length of edge $e_i$, $\bm{n}_i$ is the unit outward normal vector on $e_i$, and $|T|$ is the area of the element $T$.
\end{theorem}

From \eqref{equation.SWG}, the element stiffness matrix on $T\in\T_h$ consists of two sub-matrices corresponding to the following forms:
$$
S_T(u_b,v_b)\ \mbox{and } \B_T(u_b, v_b).
$$
The bilinear form $\B_T(\cdot,\cdot)$ is composed of three bilinear forms given by \eqref{EQ:Fulla}. The rest of this section is devoted to a computation of the element stiffness matrices for each of the bilinear forms involved.

\subsection{The stiffness matrix for $S_T(\cdot,\cdot)$}
For the element stiffness matrix corresponding to $S_T(u_b,v_b)$, the key is to compute $\S(u_b)$ and $\S(v_b)$ which can be accomplished through its definition \eqref{Def.extension}; readers are referred to \cite{LiuWang_SWG_Stokes_2018} for a detailed derivation. Specifically, let $M_T=(x_T,y_T)$ be the center of T (or any point on the plane), the extension $\S(u_b)$ can be represented as follows:
\[
\S(u_b)=\gamma_0 + \gamma_1 (x-x_T) + \gamma_2 (y-y_T),
\]
where
\begin{equation}\label{EQ:MyEq.001}
\begin{bmatrix}
\gamma_0  \\
\gamma_1  \\
\gamma_2  \\
\end{bmatrix}
=(M^TEM)^{-1}M^TE
\begin{bmatrix}
u_{b,1} \\
u_{b,2} \\
\vdots  \\
u_{b,N} \\
\end{bmatrix}.
\end{equation}
From $\S(u_b)=\gamma_0 + \gamma_1 (x-x_T) + \gamma_2 (y-y_T)$ and (\ref{EQ:MyEq.001}), we have
\begin{equation}
\begin{bmatrix}
\S(u_{b})(M_1) \\
\S(u_{b})(M_2) \\
\vdots  \\
\S(u_{b})(M_N) \\
\end{bmatrix}
=
M
\begin{bmatrix}
\gamma_0\\
\gamma_1\\
\gamma_2\\
\end{bmatrix}
=M(M^TEM)^{-1}M^TE
\begin{bmatrix}
u_{b,1} \\
u_{b,2} \\
\vdots  \\
u_{b,N} \\
\end{bmatrix}.
\end{equation}

Let $v_b\in W(T)$ be the basis function corresponding to the edge $e_j$ of $T$:
\begin{equation*}
v_b=\left\{
\begin{array}{lllll}
1, \qquad \text{ on } e_j,\\
0, \qquad \text{ otherwise}.\\
\end{array}
\right.
\end{equation*}
Then the coefficient $(\tilde{\gamma}_0,\tilde{\gamma}_1,\tilde{\gamma}_2)^T$ for $\S(v_b)$ is given by
\begin{equation*}
\begin{bmatrix}
\tilde{\gamma}_0  \\
\tilde{\gamma}_1  \\
\tilde{\gamma}_2  \\
\end{bmatrix}
=(M^TEM)^{-1}M^TE
\begin{bmatrix}
v_{b,1}\\
\vdots \\
v_{b,j}\\
\vdots \\
v_{b,N}\\
\end{bmatrix}
=(M^TEM)^{-1}M^TE
\begin{bmatrix}
0\\
\vdots \\
1\\
\vdots \\
0\\
\end{bmatrix}
\triangleq
\begin{bmatrix}
d_{1,j}\\
d_{2,j}\\
d_{3,j}\\
\end{bmatrix}.
\end{equation*}
It follows that
\begin{equation}\label{EQ:MyEQ:002}
\begin{split}
S_T(u_b,v_b)=&h^{-1}\sum_{i=1}^N(\S(u_b)(M_i)-u_{b,i})(\S(v_b)(M_i)-v_{b,i})|e_i|\\
            =&h^{-1}\sum_{i=1}^N (u_{b,i}-\S(u_b)(M_i))v_{b,i}|e_i|\\
            =&h^{-1}\left((I_N-M(M^TEM)^{-1}M^TE)
\begin{bmatrix}
u_{b,1} \\
u_{b,2} \\
\vdots  \\
u_{b,N} \\
\end{bmatrix}\right)_j |e_j| \\
 =&h^{-1}\sum_{i=1}^N a_{j,i}u_{b,i},
 \end{split}
\end{equation}
where $I_N$ is the identity matrix of size $N\times N$.

\subsection{The stiffness matrix for $a_T(\cdot,\cdot)$}
For a computation of the element stiffness matrix corresponding to the bilinear form $a_T(u_b,v_b)=(\alpha\nabla_w u_b,\nabla_w v_b )_T$, we have from the weak gradient formula \eqref{DefWGpoly} that
\begin{eqnarray*}
(\alpha\nabla_w u_b,\nabla_w v_b )_T
&=&(\alpha\displaystyle\frac{1}{|T|}\sum_{j=1}^N u_{b,j}\bm{n}_j|e_j|,\displaystyle\frac{1}{|T|}\sum_{i=1}^N v_{b,i}\bm{n}_i|e_i|)_{T}\\
&=&\sum_{i,j=1}^N(\alpha\displaystyle\frac{1}{|T|} u_{b,j}\bm{n}_j|e_j|,\displaystyle\frac{1}{|T|} v_{b,i}\bm{n}_i|e_i|)_{T} \nonumber\\
&=&\sum_{i,j=1}^N \frac{|e_j||e_i|}{|T|^2}(\alpha\bm{n}_j,\bm{n}_i)_{T}u_{b,j}v_{b,i},\nonumber \\
&=&\sum_{i,j=1}^N b_{i,j}u_{b,j}v_{b,i},\nonumber
\end{eqnarray*}
which leads to the block matrix $B$ in the element stiffness matrix.

\subsection{The stiffness matrix for $b_T(\cdot,\cdot)$}
Recall that the bilinear form $b_T(\cdot,\cdot)$ is given by
$$
b_T(u_b, v_b)=(\bbeta\cdot\nabla_w u_b, \S(v_b))_T.
$$
Note that the extension $\S(v_b)$ has the following representation:
\[
\S(v_b)=\gamma_0 + \gamma_1 (x-x_T) + \gamma_2 (y-y_T),
\]
where
\begin{equation}\label{EQ:MyEq.001-vb}
\begin{bmatrix}
\gamma_0  \\
\gamma_1  \\
\gamma_2  \\
\end{bmatrix}
=(M^TEM)^{-1}M^TE
\begin{bmatrix}
v_{b,1} \\
v_{b,2} \\
\vdots  \\
v_{b,N} \\
\end{bmatrix}.
\end{equation}
Thus, with $D=(M^TEM)^{-1}M^TE$, we have from the weak gradient formula \eqref{DefWGpoly} that
\begin{equation}\label{EQ:707}
\begin{split}
 & (\bbeta\cdot\nabla_w u_b,\S(v_b))_T \\
=&\displaystyle\frac{1}{|T|}\sum_{i,j=1}^N (\bbeta\cdot\bm{n}_j, d_{1,i} + d_{2,i} (x-x_T) + d_{3,i} (y-y_T))_T |e_j| u_{b,j} v_{b,i}\\
=&\displaystyle\frac{1}{|T|}\sum_{i,j=1}^N \int_T \bbeta\cdot\bm{n}_j (d_{1,i} + d_{2,i} (x-x_T) + d_{3,i} (y-y_T))dT |e_j| u_{b,j} v_{b,i}.
\end{split}
\end{equation}
For simplicity, we introduce the following functions:
\begin{equation}\label{EQ:zetai}
\zeta_i(x,y) = d_{1,i} + d_{2,i} (x-x_T) + d_{3,i} (y-y_T),\qquad i=1,\ldots, N.
\end{equation}
Then, the equation \eqref{EQ:707} indicates that the element stiffness matrix corresponding to the bilinear form $b_T(\cdot,\cdot)$ is given by
$$
R=\{r_{ij}\}_{N\times N},\ \ r_{ij}= \frac{|e_j|}{|T|}\int_T \bbeta\cdot\bm{n}_j \zeta_i dT.
$$

\subsection{The stiffness matrix for $c_T(\cdot,\cdot)$}
Recall that the bilinear form $c_T(\cdot,\cdot)$ is given by
$$
c_T(u_b, v_b)=(c \S(u_b), \S(v_b))_T.
$$
Thus, the element stiffness matrix corresponding to $c_T(\cdot,\cdot)$ has the following formula:
$$
C=\{c_{ij}\}_{N\times N},\quad c_{ij} = \int_T c(x,y) \zeta_j\zeta_i dT,
$$
where $\zeta_i$ is the function defined in \eqref{EQ:zetai}.

\subsection{The element load vector}

Finally, the element load vector can be obtained from
\begin{eqnarray*}
(f, \S(v_b))_T
& =& \int_{T}f \S(v_b)dT \\
& =&  \int_{T}f(x,y) (d_{1,i} + d_{2,i}(x-x_T) + d_{3,i}(y-y_T)) dT\\
& =&  \int_{T}f(x,y) \zeta_i(x,y) dT
\end{eqnarray*}
for $i=1,\ldots, N$.

\section{Stability and Well-Posedness}\label{Section:SWP}
The SWG scheme (\ref{equation.SWG}) can be derived from the classical weak Galerkin finite element method \cite{WangYe_2013,mwy,wy3655} by eliminating the degrees of freedom associated with the interior of each element when $\bbeta=0$ and $c=0$. But for the general case of $\bbeta$ and $c$, the SWG finite element method (\ref{equation.SWG}) is different from the weak Galerkin schemes in existing literature. It is thus necessary to provide a mathematical theory for the stability and well-posedness of the numerical scheme (\ref{equation.SWG}).

\begin{lemma}\label{Lemma:lemma1}
Let $\T_h$ be a shape-regular polygonal partition of the domain $\Omega$. There exists a constant $C$ such that
\begin{eqnarray}\label{EQ:815:101}
\|\nabla \S(v_b)\|_{T}^2 & \leq & C\left(\|\nabla_w v_b\|^2_T + h^{-1}\| v_b-Q_b\S(v_b)\|^2_{\pT}\right),\\
\|v_b-\S(v_b)\|_{0,\pT}^2
&\leq & C h \left(\|\nabla_w v_b\|^2_T + h^{-1} \|v_b-Q_b\S(v_b)\|^2_\pT\right).\label{EQ:815:102}
\end{eqnarray}
Moreover, the following Poincar\'e-type estimate holds true:
\begin{eqnarray}\label{EQ:815:901}
\|\S(v_b)\|^2 & \leq & C\left(\|\nabla_w v_b\|^2_T + h^{-1}\| v_b-Q_b\S(v_b)\|^2_{\pT}\right).
\end{eqnarray}
\end{lemma}

\begin{proof}
From the formula \eqref{DefWGpoly-new} for the weak gradient, we have for any constant vector $\bm{\phi}$ that
\begin{equation*}%\label{EQ:MyEQ:501}
\begin{split}
(\nabla_w v_b, \bm{\phi})_T
=&\langle v_b, \bm{\phi}\cdot\bm{n}\rangle_\pT\\
=& \langle v_b- \S(v_b), \bm{\phi}\cdot\bm{n}\rangle_\pT +  \langle \S(v_b), \bm{\phi}\cdot\bm{n}\rangle_\pT \\
=& \langle v_b- Q_b\S(v_b), \bm{\phi}\cdot\bm{n}\rangle_\pT + (\nabla \S(v_b), \bm{\phi})_T,
\end{split}
\end{equation*}
which gives
$$
(\nabla \S(v_b), \bm{\phi})_T = (\nabla_w v_b, \bm{\phi})_T - \langle v_b- Q_b\S(v_b), \bm{\phi}\cdot\bm{n}\rangle_\pT.
$$
Hence, by letting $\bm{\phi} = \nabla \S(v_b)$ we arrive at
\begin{eqnarray*}%\label{equ.EE.3NEW}
\|\nabla \S(v_b)\|_{T}^2 \leq  C\left(\|\nabla_w v_b\|^2_T + h^{-1}\| v_b-Q_b\S(v_b)\|^2_{\pT}\right),
\end{eqnarray*}
which verifies \eqref{EQ:815:101}.

Next, from the usual error estimate for the $L^2$ projection operator $Q_b$ and the estimate \eqref{EQ:815:101}, we have
\begin{equation*}%\label{equ.EE.3}
\begin{split}
\|\S(v_b) - Q_b\S(v_b)\|^2_{\pT} \leq & C h^2 \|\nabla \S(v_b)\|_\pT^2 \\
\le & C h \|\nabla \S(v_b)\|^2_T \\
\leq & C\left(h\|\nabla_w v_b\|^2_T + \|v_b-Q_b\S(v_b)\|^2_\pT\right).
\end{split}
\end{equation*}
It follows that
\begin{equation}\label{EQ:815:100}
\begin{split}
\|v_b-\S(v_b)\|_{0,\pT}
\leq & \|v_b-Q_b\S(v_b)\|_{0,\pT} + \|\S(v_b)-Q_b\S(v_b)\|_{0,\pT}\\
\leq & C \left(h\|\nabla_w v_b\|^2_T + \|v_b-Q_b\S(v_b)\|^2_\pT\right)^{1/2},
\end{split}
\end{equation}
which verifies the estimate \eqref{EQ:815:102}.

To derive the inequality \eqref{EQ:815:901}, we note the following discrete Poincar\'e inequality:
$$
\|\S(v_b)\|^2 \leq C\sum_{T\in \T_h} \left( \|\nabla \S(v_b)\|_T^2 + h_T^{-1}\|\S(v_b) - v_b\|_\pT^2\right).
$$
Combining the above estimate with \eqref{EQ:815:101} and \eqref{EQ:815:101} gives rise to the desired inequality \eqref{EQ:815:901}. This completes the proof of the lemma.
\end{proof}

\begin{lemma}\label{Lemma:lemma2}
On each element $T\in\T_h$, the following identity holds true:
\begin{equation}\label{EQ:816:001}
\begin{split}
  b_T(v_b,v_b) =&
  \frac12 \langle v_b, v_b \bbeta\cdot\bn\rangle_\pT - \frac12 ((\nabla\cdot\bbeta) \S(v_b), \S(v_b))_T \\
  & - \frac12 \langle v_b-\S(v_b), (v_b-{\S(v_b))\bbeta}\cdot\bn\rangle_\pT \\
  & + \langle v_b-\S(v_b), \overline{\S(v_b)\bbeta}\cdot\bn - {\S(v_b)\bbeta}\cdot\bn\rangle_\pT,
\end{split}
\end{equation}
where $\overline{\S(v_b)\bbeta}$ is the average of $\S(v_b)\bbeta$ on the element $T$.
\end{lemma}

\begin{proof} From the formula \eqref{DefWGpoly}, we have
\begin{equation}\label{EQ:815:211}
\begin{split}
  b_T(v_b,v_b) =& (\bbeta\cdot\nabla_w v_b, \S(v_b))_T \\
  =& (\nabla_w v_b, {\S(v_b)\bbeta})_T\\
  =& (\nabla_w v_b, \overline{\S(v_b)\bbeta})_T\\
  =& \langle v_b, \overline{\S(v_b)\bbeta}\cdot\bn\rangle_\pT\\
  =& \langle v_b-\S(v_b), \overline{\S(v_b)\bbeta}\cdot\bn\rangle_\pT
  + \langle \S(v_b), \overline{\S(v_b)\bbeta}\cdot\bn\rangle_\pT.
\end{split}
\end{equation}
Note that
\begin{equation*}
  \begin{split}
  \langle \S(v_b), \overline{\S(v_b)\bbeta}\cdot\bn\rangle_\pT &= (\nabla\S(v_b), \overline{\S(v_b)\bbeta})_T\\
  &= (\nabla\S(v_b), \S(v_b)\bbeta)_T\\
  &= \frac12 \langle \S(v_b), \S(v_b) \bbeta\cdot\bn\rangle_\pT - \frac12 ((\nabla\cdot\bbeta) \S(v_b), \S(v_b))_T.
  \end{split}
\end{equation*}
Substituting the above identity into \eqref{EQ:815:211} yields
\begin{equation}\label{EQ:815:212}
\begin{split}
  b_T(v_b,v_b) =& \langle v_b-\S(v_b), \overline{\S(v_b)\bbeta}\cdot\bn\rangle_\pT
  + \frac12 \langle \S(v_b), \S(v_b) \bbeta\cdot\bn\rangle_\pT \\
  & \ \ - \frac12 ((\nabla\cdot\bbeta) \S(v_b), \S(v_b))_T\\
  =& \langle v_b-\S(v_b), \overline{\S(v_b)\bbeta}\cdot\bn - {\S(v_b)\bbeta}\cdot\bn\rangle_\pT + \langle v_b,{\S(v_b)\bbeta}\cdot\bn\rangle_\pT \\
  & \ \ - \frac12 \langle \S(v_b), \S(v_b) \bbeta\cdot\bn\rangle_\pT - \frac12 ((\nabla\cdot\bbeta) \S(v_b), \S(v_b))_T\\
  =& \langle v_b-\S(v_b), \overline{\S(v_b)\bbeta}\cdot\bn - {\S(v_b)\bbeta}\cdot\bn\rangle_\pT \\
  & \ - \frac12 \langle v_b-\S(v_b), (v_b-{\S(v_b))\bbeta}\cdot\bn\rangle_\pT \\
  & \ + \frac12 \langle v_b, v_b \bbeta\cdot\bn\rangle_\pT - \frac12 ((\nabla\cdot\bbeta) \S(v_b), \S(v_b))_T,
\end{split}
\end{equation}
which leads to the identify \eqref{EQ:816:001}.
\end{proof}

\medskip
In the finite element space $W_h(\T_h)$, we introduce the following semi-norm:
\begin{equation}\label{EQ:815:105}
\3bar v_b\3bar^2: = \sum_{T\in\T_h} \left( \kappa S_T(v_b, v_b) + a_T(v_b, v_b) \right)
\end{equation}
We claim that $\3bar\cdot\3bar$ defines a norm in the closed subspace $W_h^0(\T_h)$. It suffices to show that $v_b \equiv 0$ for any $v_b\in W_h^0(\T_h)$ satisfying $\3bar v_b \3bar =0$. In fact, if $\3bar v_b \3bar =0$, then from \eqref{EQ:815:105} we have
\[
\kappa \sum_{T}S_T(v_b,v_b)+\sum_{T}(\alpha \nabla_w v_b,\nabla_w v_b)_T =0.
\]
It follows that on each element $T\in\T_h$
\begin{equation}\label{EQ:815-110}
\nabla_w v_b=0,\quad (v_b - \S(v_b))(M_i)=0
\end{equation}
for $i=1,\ldots, N$. Thus,
\[
\displaystyle \nabla \S(v_b)=\frac{1}{|T|}\sum_{i=1}^N \S(v_b)(M_i)|e_i|\bm{n}_i=\frac{1}{|T|}\sum_{i=1}^N v_{b,i}|e_i|\bm{n}_i=\nabla_w v_b=0,
\]
so that $\S(v_b)$ has constant value on each element $T\in \T_h$. By using \eqref{EQ:815-110} we see that $v_b=\S(v_b)=const$ on each edge, which, together with the fact that $v_b=0$ on $\partial \Omega$, leads to $v_b\equiv 0$ in $\Omega$.

\medskip

\begin{lemma}\label{Lemma:lemma4}
For the model problem \eqref{ellipticbdy}, assume that $\bbeta\in W^{1,\infty}(\Omega)$ and the condition \eqref{EQ:positive} is satisfied. Then, the bilinear form $\kappa S(\cdot,\cdot)+\B(\cdot,\cdot)$ is bounded and coercive in the finite element space $W_h^0(\T_h)$; i.e., there exist constants $M$ and $\Lambda >0$ such that
\begin{eqnarray}\label{EQ:815:200-00}
|\kappa S(v_b,w_b)+\B(v_b, w_b)| & \le & M \3bar v_b\3bar  \3bar w_b\3bar \qquad \forall v_b, w_b \in W_h^0(\T_h),\\
\label{EQ:815:200-0}
\kappa S(v_b,v_b)+\B(v_b, v_b) & \ge & \Lambda \3bar v_b\3bar^2\qquad \forall v_b\in W_h^0(\T_h),
\end{eqnarray}
provided that the meshsize $h$ of $\T_h$ is sufficiently small.
\end{lemma}

\begin{proof}
Recall that for any $v_b\in W_h^0(\T_h)$ we have
\begin{equation}\label{EQ:815:200}
\begin{split}
\B(v_b, w_b) & \ = \sum_{T\in\T_h} \left(a_T(v_b,w_b)
+ b_T(v_b, w_b) + c_T(v_b,w_b)\right),\\
S(v_b,w_b) & \ = \sum_{T\in\T_h} S_T(v_b,w_b).
\end{split}
\end{equation}
The boundedness estimate \eqref{EQ:815:200-00} is then straightforward from the usual Cauchy-Schwarz and the inequality \eqref{EQ:815:901}. We shall focus on the derivation of the coercivity inequality \eqref{EQ:815:200-0} in the rest of the proof.

In comparison with \eqref{EQ:815:105}, the key to the coercivity inequality \eqref{EQ:815:200-0} is to  derive an estimate of the following type:
\begin{equation}\label{EQ:815:201}
\sum_{T\in\T_h} \left(b_T(v_b, v_b) + c_T(v_b,v_b)\right)\ge \eta - \varepsilon(h) \3bar v_b\3bar^2,
\end{equation}
where $\eta\ge 0$ and $\varepsilon(h)$ is a parameter satisfying $\varepsilon(h) \to 0$ as $h\to 0$. If \eqref{EQ:815:201} indeed holds true, then we have from \eqref{EQ:815:200} that
\begin{equation}\label{EQ:815:250}
\begin{split}
\kappa S(v_b,v_b)+\B(v_b, v_b) \ge & \3bar v_b \3bar^2 + \eta - \varepsilon(h) \3bar v_b\3bar^2\\
\ge & (1- \varepsilon(h)) \3bar v_b\3bar^2,
\end{split}
\end{equation}
which implies the coercivity \eqref{EQ:815:200-0} for sufficiently small $h$.

It remains to derive the estimate \eqref{EQ:815:201}. To this end, we sum up the identify in Lemma \ref{Lemma:lemma2} to obtain
\begin{equation}\label{EQ:816:005}
\begin{split}
  \sum_{T\in\T_h} b_T(v_b,v_b)
  = & -\frac12 \sum_{T\in\T_h} ( \nabla\cdot\bbeta \S(v_b), \S(v_b))_T \\
  & - \frac12 \sum_{T\in\T_h}\langle v_b-\S(v_b), (v_b-{\S(v_b))\bbeta}\cdot\bn\rangle_\pT \\
  & + \sum_{T\in\T_h}\langle v_b-\S(v_b), \overline{\S(v_b)\bbeta}\cdot\bn - {\S(v_b)\bbeta}\cdot\bn\rangle_\pT,
\end{split}
\end{equation}
where we have used the fact that $\sum_{T\in\T_h} \langle v_b, v_b \bbeta\cdot\bn\rangle_\pT =0$. Thus,
\begin{equation}\label{EQ:816:006}
\begin{split}
  \sum_{T\in\T_h} (b_T(v_b,v_b) + c_T(v_b,v_b))
  = & \sum_{T\in\T_h} ( (c-\frac12 \nabla\cdot\bbeta) \S(v_b), \S(v_b))_T \\
  & - \frac12 \sum_{T\in\T_h}\langle v_b-\S(v_b), (v_b-{\S(v_b))\bbeta}\cdot\bn\rangle_\pT \\
  & + \sum_{T\in\T_h}\langle v_b-\S(v_b), \overline{\S(v_b)\bbeta}\cdot\bn - {\S(v_b)\bbeta}\cdot\bn\rangle_\pT.
\end{split}
\end{equation}
Next, from \eqref{EQ:815:102} we have
\begin{equation}\label{EQ:816:007}
\begin{split}
\left|\sum_{T\in\T_h}\langle v_b-\S(v_b), (v_b-{\S(v_b))\bbeta}\cdot\bn\rangle_\pT\right| & \leq C \left( h\|\nabla_w v_b\|^2_T + \|v_b - Q_b \S(v_b)\|_\pT^2\right) \\
& \leq C h \sum_{T\in\T_h} \left( a_T(v_b,v_b) + S_T(v_b,v_b)\right)\\
&\leq C h \3bar v_b\3bar^2.
\end{split}
\end{equation}
As to the last term in \eqref{EQ:816:006}, we have
\begin{equation}\label{EQ:816:008}
\begin{split}
& \left| \sum_{T\in\T_h}\langle v_b-\S(v_b), \overline{\S(v_b)\bbeta}\cdot\bn - {\S(v_b)\bbeta}\cdot\bn\rangle_\pT \right| \\
\leq &  \sum_{T\in\T_h} \|v_b-\S(v_b)\|_\pT \|\overline{\S(v_b)\bbeta} - {\S(v_b)\bbeta}\|_\pT\\
\leq &  C h^{\frac12} \sum_{T\in\T_h} \|v_b-\S(v_b)\|_\pT \left(\|\S(v_b)\|_T + \|\nabla \S(v_b)\|_T\right)\\
\leq &  C h \left(\sum_{T\in\T_h} h^{-1}\|v_b-\S(v_b)\|_\pT^2\right)^{\frac12} \left(\sum_{T\in\T_h} \left(\|\S(v_b)\|_T^2 + \|\nabla \S(v_b)\|_T^2\right)\right)^{\frac12}
\end{split}
\end{equation}
Combining the estimates \eqref{EQ:815:101}, \eqref{EQ:815:102}, and \eqref{EQ:815:901} with \eqref{EQ:816:008} yields
\begin{equation}\label{EQ:816:009}
\left| \sum_{T\in\T_h}\langle v_b-\S(v_b), \overline{\S(v_b)\bbeta}\cdot\bn - {\S(v_b)\bbeta}\cdot\bn\rangle_\pT \right| \leq C h \3bar v_b\3bar^2.
\end{equation}

Now by substituting \eqref{EQ:816:007} and \eqref{EQ:816:009} into \eqref{EQ:816:006} we obtain the inequality \eqref{EQ:815:201} with $\eta = ( (c-\frac12 \nabla\cdot\bbeta) \S(v_b), \S(v_b))\ge 0$ and $\varepsilon(h) = Ch$.
This completes the proof of the lemma.
\end{proof}

\medskip

The following is a direct application of Lemma \ref{Lemma:lemma4}.

%Theorem1$$$$$$$$$$$$$$$$$$$$$$$$$$$$$$$$$$$$$$$$$$$$$$$$$$$$$$$$$$$$$$$$$$$$$$$$$$$$$$$$$$$$$$$$$$$$$$$$$$$$$$$$$$$$$$$$$$$$$$$$$$$$$$$$$$$$$$$$$$$$$$$$$$$
\begin{theorem}\label{them.unique}
Under the assumptions of Lemma \ref{Lemma:lemma4}, there exists a small, but fixed number $h_0>0$, such that the numerical scheme (\ref{equation.SWG}) has one and only one solution $u_b\in W_h(\T_h)$ for sufficiently fine finite element partitions $\T_h$ satisfying $h\le h_0$.
\end{theorem}

\begin{proof}
It suffices to show that the homogeneous problem has only the trivial solution.
To this end, let $u_b \in W_h^0(\T_h)$, be the solution of scheme (\ref{equation.SWG}) with homogeneous data $f=0$ and $g=0$. By taking $v_b = u_b$ in (\ref{equation.SWG}) we obtain
\[
\kappa S(u_b,u_b)+\B(u_b,u_b)=0,
\]
which, from the coercivity inequality \eqref{EQ:815:200-0}, gives $\Lambda \3bar u_b\3bar^2 \leq \kappa S(u_b,u_b) + \B(u_b,u_b)=0$, and hence $u_b\equiv 0$ for sufficiently small $h$.
\end{proof}

\section{Error Estimates in $H^1$}\label{sectionEE}
Let $u$ be the exact solution of the model problem \eqref{ellipticbdy}-\eqref{ellipticbc} and $u_b \in W_h^0(\T_h)$ be the numerical approximation arising from the SWG scheme \eqref{equation.SWG}.
Let $Q_b u$  be the $L^2$ projection of $u$ in the space $W_h^0(\T_h)$. The error function refers to the difference between the $L^2$ projection and the SWG approximation:
\begin{equation}\label{equation.error}
e_b :=  Q_b u - u_b,
\end{equation}
The goal of this section is to establish an estimate for the error function $e_b$ in a discrete Sobolev norm.

Let us first state an error equation which plays an important role in the convergence analysis of the SWG scheme.
\begin{lemma}\label{lemma.EE.1}
Assume that the coefficient $\alpha$ of the model problem \eqref{ellipticbdy}-\eqref{ellipticbc} has piecewise constant values with respect to the finite element partition $\T_h$. Then the following equation holds true
\begin{eqnarray}\label{equ.EE.1}
%\kappa S(e_b ,v_b)+(\alpha\nabla_w e_b, \nabla_w v_b ) + (\bbeta\cdot\nabla_w e_b, \S(v_b))
\kappa S(e_b, v_b) + \B(e_b,v_b)
= \ell_u(v_b) \quad \forall v_b \in W_h^0(\T_h),
\end{eqnarray}
where $\ell_u(\cdot)$ is a linear functional given by
\begin{equation}\label{equ.EE.1-1}
\begin{split}
 \ell_u(v_b):=&\sum_{T\in\T_h} \langle \alpha\frac{\partial u}{\partial \bm{n}}-\alpha Q_0(\nabla u)\cdot\bm{n},\S(v_b)-v_b\rangle_\pT+\kappa S(Q_bu,v_b)\\
 & + ((Q_0-I) \nabla u, \S(v_b)\bbeta) + (c(\S(Q_bu)-u), \S(v_b)),
 \end{split}
\end{equation}
where $Q_0(\nabla u)$ is the $L^2$ projection of $\nabla u$ in the space $[P_0(\T_h)]^2$, and $\bm{n}$ is the outward normal vector on $\pT$.
\end{lemma}

\begin{proof}
We first consider the weak gradient of $Q_b u$, for any constant vector $\bm{\phi}$, we have
\begin{eqnarray*}
(\nabla_w Q_b u,\bm{\phi})_{T}
&=& \langle Q_b u, \bm{\phi}\cdot\bm{n} \rangle_{\partial T}
= \langle u, \bm{\phi}\cdot\bm{n} \rangle_{\partial T}\\
&=& (\nabla u, \bm{\phi})_{T}
= (Q_0(\nabla u), \bm{\phi})_{T},
\end{eqnarray*}
which implies $\nabla_w Q_b u \equiv Q_0(\nabla u)$.
Thus, for any $v_b \in W_h^0(T_h)$, we have
\begin{equation}\label{EQ:810}
\begin{split}
&(\alpha\nabla_w Q_b u,\nabla_w v_b)
= \sum_{T}(\alpha Q_0(\nabla u),\nabla_w v_b)_T \\
= & \sum_{T} \langle \alpha Q_0(\nabla u)\cdot\bm{n},v_b\rangle_{\partial T}\\
=& \sum_{T} \langle \alpha Q_0(\nabla u)\cdot\bm{n},v_b\rangle_{\partial T}-\langle\alpha Q_0(\nabla u)\cdot\bm{n},\S(v_b)\rangle_{\partial T} + \langle \alpha Q_0(\nabla u)\cdot\bm{n},\S(v_b)\rangle_{\partial T}\\
=& \sum_{T} \langle \alpha Q_0(\nabla u)\cdot\bm{n},v_b-\S(v_b)\rangle_{\partial T} +(\alpha\nabla u,\nabla \S(v_b))_{T} \\
=& \sum_{T}\langle \alpha Q_0(\nabla u)\cdot\bm{n},v_b-\S(v_b)\rangle_{\partial T} +(-\nabla\cdot(\alpha \nabla u) ,\S(v_b))_{T} + \langle\alpha\frac{\partial u}{\partial \bm{n}},\S(v_b)\rangle_{\partial T}\\
=& (-\nabla\cdot(\alpha \nabla u) ,\S(v_b))+ \sum_{T}\langle \alpha\frac{\partial u}{\partial \bm{n}}-\alpha Q_0(\nabla u)\cdot\bm{n},\S(v_b)-v_b\rangle_{\partial T}.
\end{split}
\end{equation}
Next, from $\nabla_w (Q_bu) = Q_0(\nabla u)$, we have
\begin{equation}\label{EQ:108}
\begin{split}
\sum_T (\bbeta\cdot \nabla_w (Q_bu), \S(v_b))_T = & \sum_T (\bbeta\cdot (Q_0 \nabla u), \S(v_b))_T \\
= & \sum_T (\bbeta\cdot \nabla u, \S(v_b))_T + \sum_T ((Q_0-I) \nabla u, \S(v_b)\bbeta)_T,
\end{split}
\end{equation}
and
\begin{equation}\label{EQ:109}
\sum_T (c\S(Q_bu), \S(v_b))_T = (cu, \S(v_b)) + (c(\S(Q_bu)-u), \S(v_b)).
\end{equation}
The sum of \eqref{EQ:810}, \eqref{EQ:108}, and \eqref{EQ:109} gives rise to
\begin{equation*}
\begin{split}
\B(Q_b u, v_b) = &\ (f, \S(v_b)) + \sum_{T}\langle \alpha\frac{\partial u}{\partial \bm{n}}-\alpha Q_0(\nabla u)\cdot\bm{n},\S(v_b)-v_b\rangle_{\partial T} \\
& \ + \sum_T ((Q_0-I) \nabla u, \S(v_b)\bbeta)_T + (c(\S(Q_bu)-u), \S(v_b)),
\end{split}
\end{equation*}
which, combined with $ (f, \S(v_b) = \kappa S(u_b, v_b) + \B(u_b,v_b)$, leads to
\begin{equation*}
\begin{split}
\B(Q_b u - u_b, v_b) = &\ \kappa S(u_b,v_b) + \sum_{T}\langle \alpha\frac{\partial u}{\partial \bm{n}}-\alpha Q_0(\nabla u)\cdot\bm{n},\S(v_b)-v_b\rangle_{\partial T} \\
& \ + \sum_T ((Q_0-I) \nabla u, \S(v_b)\bbeta)_T + (c(\S(Q_bu)-u), \S(v_b)),
\end{split}
\end{equation*}
and
 \begin{equation*}
\begin{split}
&\kappa S(Q_bu-u_b,v_b)+\B(Q_b u - u_b, v_b) \\
 = &\ \kappa S(Q_bu,v_b) + \sum_{T}\langle \alpha\frac{\partial u}{\partial \bm{n}}-\alpha Q_0(\nabla u)\cdot\bm{n},\S(v_b)-v_b\rangle_{\partial T} \\
& \ + \sum_T ((Q_0-I) \nabla u, \S(v_b)\bbeta)_T + (c(\S(Q_bu)-u), \S(v_b)).
\end{split}
\end{equation*}
This completes the proof of the lemma.
\end{proof}

\begin{remark}
It should be pointed out that Lemma \ref{lemma.EE.1} can be extended to the case when $\alpha$ is in $L^\infty(\Omega)$ and piecewise smooth with respect to the finite element partition $\T_h$. Detailed analysis can be established by following the approach presented in \cite{wg-systematic}.
\end{remark}

The following result is concerned with the error estimate for the SWG numerical solutions in a discrete $H^1$ norm.

\begin{theorem}\label{thm:H1}
 Let $u\in H^2(\Omega)$ be the exact solution of \eqref{ellipticbdy}-\eqref{ellipticbc} and $u_b\in W_h(\T_h)$ be the approximate solution arising from the numerical scheme (\ref{equation.SWG}). Assume $\bbeta\in C^1(\bar\Omega)$ and that \eqref{EQ:positive} is satisfied. Then, the following error estimate holds true
\begin{equation}\label{eq.H1}
\kappa S(e_b ,e_b) +(\alpha\nabla_w e_b, \nabla_w e_b )\leq Ch^2\| u\|_2^2,
\end{equation}
provided that the meshsize $h$ is sufficiently small. Consequently, we have
\begin{eqnarray}
\| \nabla_w u_b- \nabla u\|_0 \leq Ch \|u\|_2, \label{eq.EE.H1}
%\| \S(u_b)- u\|_0 \leq Ch^2 \|u\|_2. \label{eq.EE.L2}
\end{eqnarray}
\end{theorem}

\begin{proof} The proof is based on the error equation (\ref{equ.EE.1}) through a thorough analysis for the linear functional $\ell_u(\cdot)$ given in (\ref{equ.EE.1-1}). For the first term on the righ-hand side of (\ref{equ.EE.1-1}), from the usual Cauchy-Schwarz inequality we have
\begin{equation}\label{equ.EE.2}
\begin{split}
&\ |\langle\alpha\frac{\partial u}{\partial \bm{n}}-\alpha Q_0(\nabla u)\cdot\bm{n},\S(v_b)-v_b\rangle_{\partial T}| \\
\leq & \ \|\alpha\frac{\partial u}{\partial \bm{n}}-\alpha Q_0(\nabla u)\cdot\bm{n}\|_{0,\pT} \|\S(v_b)-v_b\|_{0,\pT}\\
\leq & \ \|\alpha\|_{\infty} \|\nabla u- Q_0(\nabla u)\|_{0,\pT} \|\S(v_b)-v_b\|_{0,\pT}.
\end{split}
\end{equation}
Now using the estimate \eqref{EQ:815:102} in the above inequality and then summing over all the element $T\in\T_h$ we arrive at the following:
\begin{equation}\label{eq.EE.r1}
\begin{split}
&\sum_{T\in\T_h}|\langle\alpha\frac{\partial u}{\partial \bm{n}}-\alpha Q_0(\nabla u)\cdot\bm{n},\S(v_b)-v_b\rangle_{\partial T}| \\
\leq& C \|\alpha\|_{\infty} \sum_{T\in\T_h} \|\nabla u- Q_0(\nabla u)\|_\pT \left(h\|\nabla_w v_b\|^2_T +\|v_b-Q_b\S(v_b)\|^2_\pT\right)^{\frac{1}{2}}\\
\leq & C\|\alpha\|_{\infty}\left(\|\nabla u -Q_0(\nabla u)\|^2_{0} + h^2\|\nabla^2 u \|_{0}^2\right)^{\frac{1}{2}}\left( \|\nabla_w v_b\|^2+ \kappa S(v_b, v_b)\right)^{\frac{1}{2}}\\
\leq & C h \|u\|_2 \3bar v_b \3bar.
\end{split}
\end{equation}

As to the second term on the right hand side of (\ref{equ.EE.1-1}), we have
\begin{equation}\label{eq.EE.r2}
\begin{split}
|S(Q_bu,v_b)| =&\sum_{T}h^{-1}\langle Q_bu-Q_b\S(Q_bu),v_b-Q_b \S(v_b)\rangle_{\partial T} \\
=&\sum_{T}h^{-1}\langle Q_bu,v_b-Q_b \S(v_b)\rangle_{\partial T} \\%\text{ By orthogonality}\\
=&\sum_{T}h^{-1}\langle Q_bu -Q_b(Q_1 u),v_b-Q_b \S(v_b)\rangle_{\partial T}\\
=&\sum_{T}h^{-1}\langle u -Q_1 u,v_b-Q_b \S(v_b)\rangle_{\partial T}\\
\leq &\left(\sum_{T} h^{-1} \int_{\partial T}|u-Q_1u|^2ds \right)^{\frac{1}{2}}S(v_b, v_b)^{\frac{1}{2}}\\
\leq & C\left(h^{-2}\|u-Q_1 u\|^2 + \|u-Q_1 u\|_1^2 \right)^{\frac{1}{2}}S(v_b, v_b)^{\frac{1}{2}}\\
\leq & C h \|u\|_2 \3bar v_b\3bar.
\end{split}
\end{equation}

The third term on the right hand side of (\ref{equ.EE.1-1}) can be bounded by using the usual error estimate for $L^2$ projections as follows:
\begin{equation}\label{EQ:401}
\begin{split}
|((Q_0-I) \nabla u, \S(v_b)\bbeta)| = &  |((Q_0-I) \nabla u, (Q_0-I) (\S(v_b)\bbeta))| \\
\le & \|(Q_0-I) \nabla u\| \ \|(Q_0-I) (\S(v_b)\bbeta)\| \\
\le & C h^2 \|\nabla^2 u\| \left( \|\nabla \S(v_b)\| + \|\S(v_b)\|\right) \\
\le & C h^2 \|\nabla^2 u\| \3bar v_b\3bar,
\end{split}
\end{equation}
where we have used the estimates \eqref{EQ:815:101} and \eqref{EQ:815:901} in the last line.

The last term on the right hand side of (\ref{equ.EE.1-1}) can be estimated as follows:
\begin{equation}\label{EQ:402}
\begin{split}
|(c(\S(Q_bu)-u), \S(v_b))|\leq & \|c\|_\infty \|\S(Q_bu)-u\| \|\S(v_b)\|\\
\leq & C (\|\S(Q_bu)-Q_1 u\| + \|Q_1u - u\|) \|\S(v_b)\|\\
\leq & C ( \|\S(Q_bu)-\S(Q_1 u)\| + \|Q_1u - u\|) \|\S(v_b)\|\\
\leq & C ( \|\S(Q_bu-Q_1 u)\| + \|Q_1u - u\|) \|\S(v_b)\|\\
\leq & C h^2 \|u\|_2 \3bar v_b\3bar.
\end{split}
\end{equation}

Substituting the estimates \eqref{eq.EE.r1}-\eqref{EQ:402} into the error equation \eqref{eq.EE.r1} yields
$$
\kappa S(e_b, v_b) + \B(e_b, v_b) \leq C h \|u\|_2 \3bar v_b\3bar,
$$
which, together with the coercivity \eqref{EQ:815:200-0}, leads to
$$
\Lambda \3bar e_b\3bar^2 \leq C h \|u\|_2 \3bar e_b\3bar.
$$
The last inequality implies the error estimate \eqref{eq.H1}.

Finally, from the triangle inequality and the error estimate \eqref{eq.H1}, we obtain
\begin{equation*}
\begin{split}
\|\nabla_w u_b - \nabla u\| \leq & \|\nabla_w (u_b - Q_b u)\| + \|\nabla_w (Q_b u) - \nabla u\|\\
 = & \|\nabla_w e_b\| + \|Q_0 (\nabla u) - \nabla u\|\\
 \leq & C h \|u\|_2,
 \end{split}
 \end{equation*}
which gives rise to \eqref{eq.EE.H1}. This completes the proof of the theorem.
\end{proof}

\section{Error Estimates in $L^2$}\label{sectionEE-L2} We use the usual duality argument to derive an error estimate in $L^2$ for the numerical solutions arising from \eqref{equation.SWG}. The analysis to be presented is a modified version of those developed in \cite{WangYe_2013, mwy, wg-systematic}.

Consider the following auxiliary problem that seeks $\Phi\in H_0^1(\Omega)$ such that
\begin{eqnarray}
-\nabla\cdot(\alpha\nabla \Phi) - \nabla\cdot(\bbeta\Phi) + c\Phi &=&\chi\quad {\rm in}\  \Omega  \label{dual-problem}\\
\Phi&=&0\quad {\rm on}\ \partial\Omega, \label{dualbc}
\end{eqnarray}
where $\chi\in L^2(\Omega)$. Assume that the solution of the problem \eqref{dual-problem}-\eqref{dualbc} exists and has the $H^2$-regularity:
\begin{equation}\label{dual-regularity}
\|\Phi\|_2 \leq C \|\chi\|,
\end{equation}
where $C$ is a constant depending only on the domain and the coefficients $\alpha, \bbeta$, and $c$.

\begin{theorem}\label{thm:L2}
 Let $u\in H^2(\Omega)$ be the exact solution of \eqref{ellipticbdy}-\eqref{ellipticbc} and $u_b\in W_h(\T_h)$ be the approximate solution arising from the numerical scheme (\ref{equation.SWG}). Assume $\bbeta\in C^1(\bar\Omega)$ and the conditions \eqref{EQ:positive} and \eqref{dual-regularity} are satisfied. Then, the following $L^2$ error estimate holds true
\begin{equation}\label{eq.L2}
\|u - \S(u_b)\| \le Ch^2\|u\|_2,
\end{equation}
provided that the meshsize $h$ is sufficiently small.
\end{theorem}

\begin{proof}
On each element $T\in\T_h$, we test \eqref{dual-problem} against the linear function $\S(e_b)$ to obtain
\begin{eqnarray*}
(\chi, \S(e_b))_T &=& (\alpha\nabla\Phi, \nabla\S(e_b))_T + (\bbeta\Phi, \nabla\S(e_b))_T + (c\Phi, \S(e_b))_T \\
&& - \langle \alpha\nabla\Phi\cdot\bn, \S(e_b)\rangle_\pT -\langle \bbeta\cdot\bn \Phi, \S(e_b)\rangle_\pT\\
&=& (\alpha Q_0(\nabla\Phi), \nabla\S(e_b))_T + (Q_0(\bbeta\Phi), \nabla\S(e_b))_T + (c\Phi, \S(e_b))_T \\
&& - \langle \alpha\nabla\Phi\cdot\bn, \S(e_b)\rangle_\pT -\langle \bbeta\cdot\bn \Phi, \S(e_b)\rangle_\pT\\
&=& (\alpha Q_0(\nabla\Phi), \nabla_w e_b)_T + (Q_0(\bbeta\Phi), \nabla_w e_b)_T + (c\Phi, \S(e_b))_T \\
&& - \langle \alpha\nabla\Phi\cdot\bn, \S(e_b)\rangle_\pT -\langle \bbeta\cdot\bn \Phi, \S(e_b)\rangle_\pT \\
&& - \langle \alpha Q_0(\nabla\Phi)\cdot\bn, e_b-\S(e_b)\rangle_\pT - \langle Q_0(\bbeta\Phi)\cdot\bn, e_b-\S(e_b)\rangle_\pT
\end{eqnarray*}
By using $Q_0(\nabla \Phi) = \nabla_w (Q_b\Phi)$ and $(Q_0(\bbeta\Phi), \nabla_w e_b)_T = (\bbeta\cdot \nabla_w e_b,  \Phi)_T$ in the above equation, we have from summing over all $T\in\T_h$ that
\begin{equation}\label{EQ:818:100}
\begin{split}
(\chi, \S(e_b))
=& (\alpha \nabla_w e_b, \nabla_w (Q_b\Phi)) + (\bbeta\cdot\nabla_w e_b, \Phi) + (c\S(e_b), \Phi) \\
& - \sum_T \langle \alpha\nabla \Phi\cdot\bn - \alpha Q_0(\nabla \Phi)\cdot\bn, \S(e_b)-e_b\rangle_\pT\\
& -\langle \bbeta\Phi\cdot\bn - Q_0(\bbeta\Phi)\cdot\bn , \S(e_b)-e_b\rangle_\pT.
\end{split}
\end{equation}
The last two terms on the right-hand side of \eqref{EQ:818:100} can be bounded by $C h \|\Phi\|_2 \3bar e_b\3bar$ through the Cauchy-Schwarz inequality. Thus, we have
\begin{equation}\label{EQ:818:101}
\begin{split}
|(\chi, \S(e_b))|
\leq & |(\alpha \nabla_w e_b, \nabla_w (Q_b\Phi)) + (\bbeta\cdot\nabla_w e_b, \Phi) + (c\S(e_b), \Phi)| \\
& + Ch\|\Phi\|_2 \3bar e_b\3bar\\
\leq & |(\alpha \nabla_w e_b, \nabla_w (Q_b\Phi)) + (\bbeta\cdot\nabla_w e_b, \S(Q_b\Phi)) + (c\S(e_b), \S(Q_b\Phi))| \\
& + Ch\|\Phi\|_2 \3bar e_b\3bar,
\end{split}
\end{equation}
where have also used $\|\Phi-\S(Q_b\Phi)\| \leq C h^2\|\Phi\|_2$. Now, recall that
$$
(\alpha \nabla_w e_b, \nabla_w (Q_b\Phi)) + (\bbeta\cdot\nabla_w e_b, \S(Q_b\Phi)) + (c\S(e_b), \S(Q_b\Phi))
= \B(e_b, Q_b\Phi),
$$
and from the error equation \eqref{equ.EE.1}, we have
\begin{equation}\label{EQ:818:200}
\begin{split}
\B(e_b, Q_b\Phi) = & \ell_u(Q_b\Phi) - \kappa S(e_b, Q_b\Phi)\\
=&\sum_{T\in\T_h} \langle \alpha\frac{\partial u}{\partial \bm{n}}-\alpha Q_0(\nabla u)\cdot\bm{n},\S(Q_b\Phi)-Q_b\Phi\rangle_\pT+\kappa S(u_b,Q_b\Phi)\\
 & + ((Q_0-I) \nabla u, \S(Q_b\Phi)\bbeta) + (c(\S(Q_bu)-u), \S(Q_b\Phi)),
 \end{split}
\end{equation}

The last two terms on the right-hand side of \eqref{EQ:818:200} have the following estimate:
\begin{equation}\label{EQ:818:102-01}
\left|((Q_0-I) \nabla u, \S(Q_b\Phi)\bbeta) + (c(\S(Q_bu)-u), \S(Q_b\Phi))\right| \leq C h^2 \|u\|_2 \|\Phi\|_1.
\end{equation}

The second term, $\kappa S(u_b,Q_b\Phi)$, can be dealt with as follows:
\begin{equation}\label{EQ:818:102}
\begin{split}
\kappa S(u_b,Q_b\Phi) = & \kappa h^{-1} \sum_T \langle u_b-Q_b\S(u_b), Q_b\Phi - Q_b\S(Q_b\Phi)\rangle_\pT\\
= & \kappa h^{-1} \sum_T \langle u_b-Q_b\S(u_b), \Phi - \S(Q_b\Phi)\rangle_\pT\\
\leq & \kappa h^{-1} \sum_T \|u_b-Q_b\S(u_b)\|_\pT \|\Phi - \S(Q_b\Phi)\|_\pT\\
\leq & C h (\3bar e_b\3bar + h\|u\|_2)\|\Phi\|_2.
\end{split}
\end{equation}

As to the first term, we note from the definition of $Q_b$ and $\Phi|_{\partial\Omega}=0$ that
\begin{equation*}
\begin{split}
\sum_{T\in\T_h} \langle \alpha\frac{\partial u}{\partial \bm{n}}-\alpha Q_0(\nabla u)\cdot\bm{n},\Phi-Q_b\Phi\rangle_\pT
= \sum_{T\in\T_h} \langle \alpha\frac{\partial u}{\partial \bm{n}},\Phi-Q_b\Phi\rangle_\pT
= 0.
\end{split}
\end{equation*}
Thus, we have
\begin{equation}\label{EQ:818:105}
\begin{split}
&\sum_{T\in\T_h} \langle \alpha\frac{\partial u}{\partial \bm{n}}-\alpha Q_0(\nabla u)\cdot\bm{n},\S(Q_b\Phi)-Q_b\Phi\rangle_\pT \\
= & \sum_{T\in\T_h} \langle \alpha\frac{\partial u}{\partial \bm{n}}-\alpha Q_0(\nabla u)\cdot\bm{n},\S(Q_b\Phi)-\Phi\rangle_\pT\\
\le & C h^2  \|u\|_2\|\Phi\|_2.
\end{split}
\end{equation}
Substituting \eqref{EQ:818:102-01}, \eqref{EQ:818:102}, and \eqref{EQ:818:105} into \eqref{EQ:818:200} yields the following estimate:
$$
|\B(e_b, Q_b\Phi)| \leq C(h^2 \|u\|_2 + h\3bar e_b\3bar) \|\Phi\|_2,
$$
which, together with \eqref{EQ:818:101}, leads to
\begin{equation}\label{EQ:818:108}
|(\chi, \S(e_b))| \leq C(h^2 \|u\|_2 + h\3bar e_b\3bar) \|\Phi\|_2 \leq C(h^2 \|u\|_2 + h\3bar e_b\3bar) \|\chi\|,
\end{equation}
where the regularity assumption \eqref{dual-regularity} has been employed in the last inequality.

Next, from \eqref{EQ:818:108} and the $H^1$ error estimate \eqref{eq.H1} in Theorem \ref{thm:H1}, we have
\begin{equation*}
|(\chi, \S(e_b))| \leq Ch^2 \|u\|_2 \|\chi\|,
\end{equation*}
which leads to
$$
\|\S(e_b\| \leq C h^2 \|u\|_2.
$$
Finally, we arrive at
$$
\|u-\S(u_b)\|\leq \|u-\S(Q_bu)\| + \|\S(e_b)\|\leq C h^2\|u\|_2,
$$
which completes the proof of the theorem.
\end{proof}

\section{Numerical Experiments}\label{numerical-experiments}
The goal of this section is to numerically verify the error estimates developed in the previous sections for the numerical scheme (\ref{equation.SWG}).
The following metrics are employed to measure the magnitude of the error function:
\begin{eqnarray*}
&&\text{Discrete $L^2$-norm: }\\
&&\|u_b-u \|_{0}=h\left(\sum_{i=1}^{n+1}\sum_{j=1}^{n}|u_{i-\frac{1}{2},j}-u(x_{i-\frac{1}{2}},y_j)|^2  + \sum_{i=1}^{n}\sum_{j=1}^{n+1}|u_{i,j-\frac{1}{2}}-u(x_{i},y_{j-\frac{1}{2}})|^2\right)^{1/2},\\
&&\text{Discrete $H^1$-norm: }\\
&&\|u_b-u\|_1= h \left(\sum_{i=1}^{n}\sum_{j=1}^{n} \left|\frac{u_{i+\frac{1}{2},j}-u_{i-\frac{1}{2},j}}{h} - \frac{\partial u}{\partial x} (x_{i},y_{j})\right|^2 \right. \\
&&\qquad\qquad\qquad \left. + \sum_{i=1}^{n}\sum_{j=1}^{n} \left|\frac{u_{i,j+\frac{1}{2}}-u_{i,j-\frac{1}{2}}}{h} - \frac{\partial u}{\partial y} (x_{i},y_{j})\right|^2  \right)^{1/2} ,\\
\end{eqnarray*}
%For simplicity, we use only uniform square partitions for the unit square domain $\Omega=(0,1)\times (0,1)$ in the numerical experiments.

Our numerical experiments are conducted for the model problem \eqref{ellipticbdy}-\eqref{ellipticbc} on polygonal domains.
The following set of test cases are considered:
\begin{equation}\label{eq.testcase.1}
\left\{
\begin{split}
&u=xy, \\
&\alpha=
\begin{bmatrix}
 1 &0\\
 0 & 1
\end{bmatrix},
\quad \bm{\beta}=
\begin{bmatrix}
1\\
1
\end{bmatrix},
\quad c=1;
\end{split}
\right.
\end{equation}

\begin{equation}\label{eq.testcase.2}
\left\{
\begin{split}
&u= 3x^2+2xy, \\
&\alpha=
\begin{bmatrix}
 2 &0\\
 0 & 1
\end{bmatrix},
\quad \bm{\beta}=
\begin{bmatrix}
1\\
1
\end{bmatrix},
\quad c=1;
\end{split}
\right.
\end{equation}

\begin{equation}\label{eq.testcase.3}
\left\{
\begin{split}
&u=\sin(\pi x)\sin(\pi y) + x^2-y^2,  \\
&\alpha=
\begin{bmatrix}
 1 &0\\
 0 & 1
\end{bmatrix},
\quad \bm{\beta}=
\begin{bmatrix}
1\\
2
\end{bmatrix},
\quad c=1;
\end{split}
\right.
\end{equation}

\begin{equation}\label{eq.testcase.4}
\left\{
\begin{split}
&u=\sin(\pi x)\sin(\pi y),  \\
&\alpha=
\begin{bmatrix}
 xy+1 &0\\
 0 & 3xy
\end{bmatrix},
\quad \bm{\beta}=
\begin{bmatrix}
x^3y+xy+1\\
3x^2y+xy+2
\end{bmatrix},
\quad c=x^4y^2+xy+1;
\end{split}
\right.
\end{equation}
The right-hand side function $f$ and the Dirichlet boundary data $g$ are chosen to match the exact solution $u=u(x,y)$ for each test case.

\begin{table}[h!]
{\small
\begin{center}
\caption{Error and convergence performance of the SWG scheme \eqref{equation.SWG} with $\kappa=4.0$ and uniform square partitions on the unit square domain $\Omega=(0,1)^2$.}\label{tab.testcase1}
\begin{tabular}{||c|cc|cc|cc|cc||}
\hline
     & \multicolumn{4}{| c |}{Test case \eqref{eq.testcase.1}} & \multicolumn{4}{| c |}{Test case \eqref{eq.testcase.2}}\\
\hline
$h^{-1}$ & $\|u_h-u \|_{0}$ & Rate & $\|u_h-u\|_{1}$ & Rate & $\|u_h-u \|_{0}$ & Rate & $\|u_h-u\|_{1}$ & Rate\\
\hline
8  &    2.92e-16&  - &    1.38e-15 & - &    1.32e-02  &    -   &  4.57e-02  &    -\\
16 &    2.86e-15&  - &    1.02e-14 & - &    3.36e-03  & 1.98   &  1.28e-02  & 1.84\\
32 &    1.00e-14&  - &    3.63e-14 & - &    8.43e-04  & 1.99   &  3.49e-03  & 1.87\\
64 &    4.10e-14&  - &    1.48e-13 & - &    2.11e-04  & 2.00   &  9.43e-04  & 1.89\\
128&    1.66e-13&  - &    5.96e-13 & - &    5.28e-05  & 2.00   &  2.52e-04  & 1.90\\
\hline
     & \multicolumn{4}{| c |}{Test case \eqref{eq.testcase.3}} & \multicolumn{4}{| c |}{Test case \eqref{eq.testcase.4}}\\
\hline
$h^{-1}$ & $\|u_h-u \|_{0}$ & Rate & $\|u_h-u\|_{1}$ & Rate & $\|u_h-u \|_{0}$ & Rate & $\|u_h-u\|_{1}$ & Rate\\
\hline
8  &    1.97e-02  &    -   &  4.19e-02  &    -    & 2.59e-02  &    -  &   6.94e-02  &    -\\
16 &    4.93e-03  & 2.00   &  1.05e-02  & 2.00    & 6.48e-03  & 2.00  &   1.76e-02  & 1.98\\
32 &    1.23e-03  & 2.00   &  2.63e-03  & 2.00    & 1.62e-03  & 2.00  &   4.43e-03  & 1.99\\
64 &    3.08e-04  & 2.00   &  6.58e-04  & 2.00    & 4.06e-04  & 2.00  &   1.11e-03  & 1.99\\
128&    7.69e-05  & 2.00   &  1.65e-04  & 2.00    & 1.02e-04  & 2.00  &   2.79e-04  & 2.00\\
\hline
\end{tabular}
\end{center}
}
\end{table}

Table \ref{tab.testcase1} shows the performance of the SWG scheme for each of the above test problems with the stabilizer parameter $\kappa=4$ on uniform square partitions. The results indicate that the numerical approximation is in the machine accuracy for the test problem \eqref{eq.testcase.1} where the exact solution is a bilinear function. For the other three test problems, the numerical solutions have the optimal rate of convergence $r=2$ in the discrete $L^2$ norm and a superconvergence of order $\O(h^2)$ in the discrete $H^1$ norm. The numerical results are consistent with the theoretical prediction in the discrete $L^2$ norm, but they outperform the theory in the discrete $H^1$ norm. It should be pointed out that the superconvergence theory in \cite{LiDanWW} was developed for the diffusion equation only; but a slight modification of the analysis there will yield a superconvergence of order $\O(h^2)$ for the SWG solutions of the full convection-diffusion equation \eqref{ellipticbdy}-\eqref{ellipticbc}.

\begin{table}[h]
{\small
\begin{center}
\caption{Error and convergence performance of the SWG scheme \eqref{equation.SWG} for the test case \eqref{eq.testcase.3} with different values of $\kappa$ on uniform square partitions for $\Omega=(0,1)^2$.}\label{tab.testcase2}
\begin{tabular}{||c|cc|cc|cc|cc||}
\hline
     & \multicolumn{4}{| c |}{$\kappa=0.01$} & \multicolumn{4}{| c |}{$\kappa=0.1$}\\
\hline
$h^{-1}$ & $\|u_h-u \|_{0}$ & Rate & $\|u_h-u\|_{1}$ & Rate & $\|u_h-u \|_{0}$ & Rate & $\|u_h-u\|_{1}$ & Rate\\
\hline
8   &3.30e-01 &     -   &  1.04e+00 &     - &    1.70e-01 &     -  &   5.38e-01 &     -\\
16  &2.50e-01 &  0.40   &  7.97e-01 &  0.39 &    6.67e-02 &  1.35  &   2.16e-01 &  1.31\\
32  &1.30e-01 &  0.94   &  4.19e-01 &  0.93 &    1.98e-02 &  1.75  &   6.74e-02 &  1.68\\
64  &4.59e-02 &  1.51   &  1.51e-01 &  1.47 &    5.23e-03 &  1.92  &   1.92e-02 &  1.81\\
128 &1.29e-02 &  1.83   &  4.52e-02 &  1.74 &    1.33e-03 &  1.98  &   5.31e-03 &  1.86\\
\hline
     & \multicolumn{4}{| c |}{$\kappa=1.0$} & \multicolumn{4}{| c |}{$\kappa=4.0$}\\
\hline
$h^{-1}$ & $\|u_h-u \|_{0}$ & Rate & $\|u_h-u\|_{1}$ & Rate & $\|u_h-u \|_{0}$ & Rate & $\|u_h-u\|_{1}$ & Rate\\
\hline
8  &    3.11e-02 &     -  &   8.97e-02 &     - &    1.97e-02  &    -   &  4.19e-02  &    -\\
16 &    8.12e-03 &  1.94  &   2.53e-02 &  1.83 &    4.93e-03  & 2.00   &  1.05e-02  & 2.00\\
32 &    2.06e-03 &  1.98  &   6.91e-03 &  1.87 &    1.23e-03  & 2.00   &  2.63e-03  & 2.00\\
64 &    5.16e-04 &  1.99  &   1.86e-03 &  1.89 &    3.08e-04  & 2.00   &  6.58e-04  & 2.00\\
128&    1.29e-04 &  2.00  &   4.96e-04 &  1.91 &    7.69e-05  & 2.00   &  1.65e-04  & 2.00\\
\hline
     & \multicolumn{4}{| c |}{$\kappa=6.0$} & \multicolumn{4}{| c |}{$\kappa=20.0$}\\
\hline
$h^{-1}$ & $\|u_h-u \|_{0}$ & Rate & $\|u_h-u\|_{1}$ & Rate & $\|u_h-u \|_{0}$ & Rate & $\|u_h-u\|_{1}$ & Rate\\
\hline
8  &    1.99e-02 &     -  &   4.30e-02  &    - &    2.09e-02  &    -   &  4.96e-02  &    -\\
16 &    4.97e-03 &  2.00  &   1.08e-02  & 1.99 &    5.20e-03  & 2.01   &  1.28e-02  & 1.96\\
32 &    1.24e-03 &  2.00  &   2.73e-03  & 1.99 &    1.30e-03  & 2.00   &  3.27e-03  & 1.96\\
64 &    3.10e-04 &  2.00  &   6.87e-04  & 1.99 &    3.25e-04  & 2.00   &  8.39e-04  & 1.96\\
128&    7.76e-05 &  2.00  &   1.73e-04  & 1.99 &    8.12e-05  & 2.00   &  2.15e-04  & 1.97\\
\hline
\end{tabular}
\end{center}
}
\end{table}

\subsection{On the influence of the stabilizer parameter}
The goal of this subsection is to test the influence of the stabilizer parameter $\kappa$ on the numerical solutions. This part of the numerical experiment considers only the test cases \eqref{eq.testcase.3} and \eqref{eq.testcase.4} with the following six values of $\kappa=0.01, \ 0.1, \ 1.0, \ 4.0, \ 6.0, \ 20.0$.
The case of $\kappa=0$ is not a viable choice , as it was not covered in the convergence theory. In fact, our computation does not suggest any convergence of the scheme when $\kappa=0$.

Tables \ref{tab.testcase2}-\ref{tab.testcase3} illustrate the numerical performance of the SWG scheme with different values of the stabilizer parameter $\kappa$. Note that, for both test cases, the rate of convergence deteriorates as $\kappa$ gets small (e.g. $\kappa=0.01$), particulary on coarse finite element partitions, but the rate of convergence begins to improve when the meshsize $h$ gets small. Optimal rate of convergence and the supercovergence of order $\O(h^2)$ are clearly shown in the tables when $\kappa$ is away from $0$ (e.g., $\kappa\ge 0.1$). The stability and accuracy of the SWG scheme is insensitive to the value of $\kappa$ as long as it stays away from $0$.

\begin{table}[h]
{\small
\begin{center}
\caption{Error and convergence performance of the SWG scheme \eqref{equation.SWG} for the test case \eqref{eq.testcase.4} with different values of $\kappa$ on uniform square partitions for $\Omega=(0,1)^2$.}\label{tab.testcase3}
\begin{tabular}{||c|cc|cc|cc|cc||}
\hline
     & \multicolumn{4}{| c |}{$\kappa=0.01$} & \multicolumn{4}{| c |}{$\kappa=0.1$}\\
\hline
$h^{-1}$ & $\|u_h-u \|_{0}$ & Rate & $\|u_h-u\|_{1}$ & Rate & $\|u_h-u \|_{0}$ & Rate & $\|u_h-u\|_{1}$ & Rate\\
\hline
 8   & 6.16e-01  &    -   &  2.12e+00 &     - &   2.45e-01  &    -   &  8.59e-01  &    -\\
 16  & 4.01e-01  & 0.62   &  1.47e+00 &  0.53 &   8.26e-02  & 1.57   &  3.06e-01  & 1.49\\
 32  & 1.70e-01  & 1.24   &  6.49e-01 &  1.18 &   2.31e-02  & 1.84   &  8.86e-02  & 1.79\\
 64  & 5.30e-02  & 1.68   &  2.09e-01 &  1.64 &   5.99e-03  & 1.95   &  2.35e-02  & 1.91\\
 128 & 1.44e-02  & 1.88   &  5.78e-02 &  1.85 &   1.51e-03  & 1.98   &  6.03e-03  & 1.96\\
 \hline
     & \multicolumn{4}{| c |}{$\kappa=1.0$} & \multicolumn{4}{| c |}{$\kappa=4.0$}\\
\hline
$h^{-1}$ & $\|u_h-u \|_{0}$ & Rate & $\|u_h-u\|_{1}$ & Rate & $\|u_h-u \|_{0}$ & Rate & $\|u_h-u\|_{1}$ & Rate\\
\hline
 8  &   4.80e-02  &    -  &   1.56e-01  &    - &    2.59e-02  &    - &    6.94e-02 &     -\\
 16 &   1.23e-02  & 1.96  &   4.15e-02  & 1.91 &    6.48e-03  & 2.00 &    1.76e-02 &  1.98\\
 32 &   3.10e-03  & 1.99  &   1.07e-02  & 1.96 &    1.62e-03  & 2.00 &    4.43e-03 &  1.99\\
 64 &   7.78e-04  & 2.00  &   2.71e-03  & 1.98 &    4.06e-04  & 2.00 &    1.11e-03 &  1.99\\
 128&   1.95e-04  & 2.00  &   6.84e-04  & 1.98 &    1.02e-04  & 2.00 &    2.79e-04 &  2.00\\
 \hline
     & \multicolumn{4}{| c |}{$\kappa=6.0$} & \multicolumn{4}{| c |}{$\kappa=20.0$}\\
\hline
$h^{-1}$ & $\|u_h-u \|_{0}$ & Rate & $\|u_h-u\|_{1}$ & Rate & $\|u_h-u \|_{0}$ & Rate & $\|u_h-u\|_{1}$ & Rate\\
\hline
8  &    2.39e-02 &     -  &   6.11e-02  &    - &    2.17e-02 &     -  &   5.17e-02 &     -\\
16 &    5.99e-03 &  2.00  &   1.55e-02  & 1.98 &    5.41e-03 &  2.00  &   1.31e-02 &  1.98\\
32 &    1.50e-03 &  2.00  &   3.89e-03  & 1.99 &    1.35e-03 &  2.00  &   3.28e-03 &  1.99\\
64 &    3.75e-04 &  2.00  &   9.76e-04  & 1.99 &    3.38e-04 &  2.00  &   8.23e-04 &  2.00\\
128&    9.37e-05 &  2.00  &   2.44e-04  & 2.00 &    8.46e-05 &  2.00  &   2.06e-04 &  2.00\\
\hline
\end{tabular}
\end{center}
}
\end{table}

\subsection{SWG with general polygonal partitions}
The SWG scheme was applied to the test problem \eqref{eq.testcase.3} with general polygonal partitions. Table \ref{tab.testcase4} shows the error and convergence performance of the scheme on four types of polygonal partitions. The stabilization parameter was set as $\kappa=4$ in all these tests. Optimal order of convergence in the discrete $L^2$ norm can be observed for each polygonal partition, but the superconvergence in the discrete $H^1$ norm was only seen for rectangular partitions. The table shows a numerical rate of convergence of $r=1$ in the discrete $H^1$ norm for three other type of partitions. The result is clearly in consistency with the error estimate developed in Section \ref{sectionEE}.

Fig. \ref{fig.comp_solution} illustrates the contour plots of the numerical solutions on different type of polygonal partitions. It also shows the shape of the polygonal elements in our computation.

\begin{table}[h]
{\small
\begin{center}
\caption{Error and convergence performance of the SWG scheme \eqref{equation.SWG} for the test problem \eqref{eq.testcase.3} on general polygonal partitions for $\Omega=(0,1)^2$, with $\kappa = 4$.}\label{tab.testcase4}
\begin{tabular}{||c|cc|cc|cc|cc||}
\hline
     & \multicolumn{4}{| c |}{Triangular mesh} & \multicolumn{4}{| c |}{Rectangular mesh}\\
\hline
$h^{-1}$ & $\|u_h-u \|_{0}$ & Rate & $\|u_h-u\|_{1}$ & Rate & $\|u_h-u \|_{0}$ & Rate & $\|u_h-u\|_{1}$ & Rate\\
\hline
 8  &   1.29e-02 &     -  &   2.53e-01 &     -   &  1.97e-02 &     -  &   4.19e-02 &     - \\
 16 &   3.25e-03 &  1.99  &   1.27e-01 &  0.99   &  4.93e-03 &  2.00  &   1.05e-02 &  2.00 \\
 32 &   8.16e-04 &  2.00  &   6.35e-02 &  1.00   &  1.23e-03 &  2.00  &   2.63e-03 &  2.00 \\
 64 &   2.04e-04 &  2.00  &   3.17e-02 &  1.00   &  3.08e-04 &  2.00  &   6.58e-04 &  2.00 \\
 128&   5.10e-05 &  2.00  &   1.59e-02 &  1.00   &  7.69e-05 &  2.00  &   1.65e-04 &  2.00 \\
\hline
     & \multicolumn{4}{| c |}{Hexagonal mesh} & \multicolumn{4}{| c |}{Octagonal mesh}\\
\hline
$h^{-1}$ & $\|u_h-u \|_{0}$ & Rate & $\|u_h-u\|_{1}$ & Rate & $\|u_h-u \|_{0}$ & Rate & $\|u_h-u\|_{1}$ & Rate\\
\hline
 8  &   1.38e-02 &     -   &  8.27e-02  &    - &    2.31e-02  &    -  &   8.19e-02 &     -\\
 16 &   3.34e-03 &  2.04   &  4.00e-02  & 1.05 &    5.83e-03  & 1.98  &   4.08e-02 &  1.00\\
 32 &   8.69e-04 &  1.94   &  2.06e-02  & 0.96 &    1.55e-03  & 1.91  &   1.99e-02 &  1.03\\
 64 &   2.21e-04 &  1.98   &  1.05e-02  & 0.97 &    3.61e-04  & 2.10  &   1.01e-02 &  0.98\\
 128&   5.52e-05 &  2.00   &  5.32e-03  & 0.98 &    9.50e-05  & 1.92  &   5.08e-03 &  0.99\\
 \hline
\end{tabular}
\end{center}
}
\end{table}

\begin{figure}[h]
\includegraphics [width=1.0\textwidth]{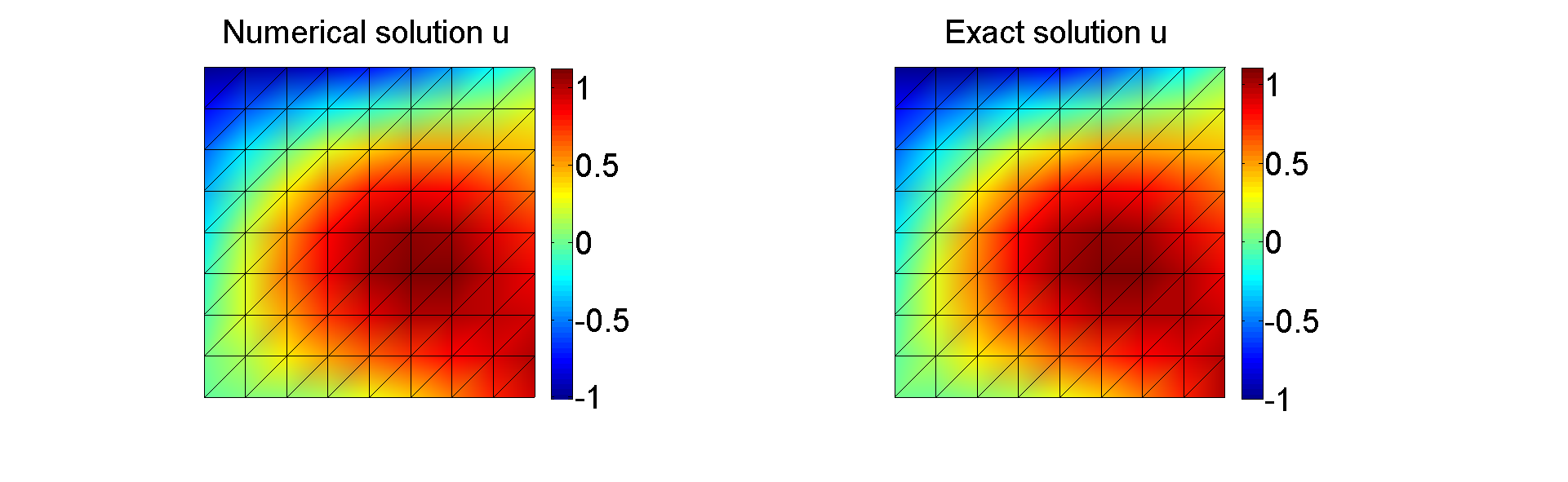}
\includegraphics [width=1.0\textwidth]{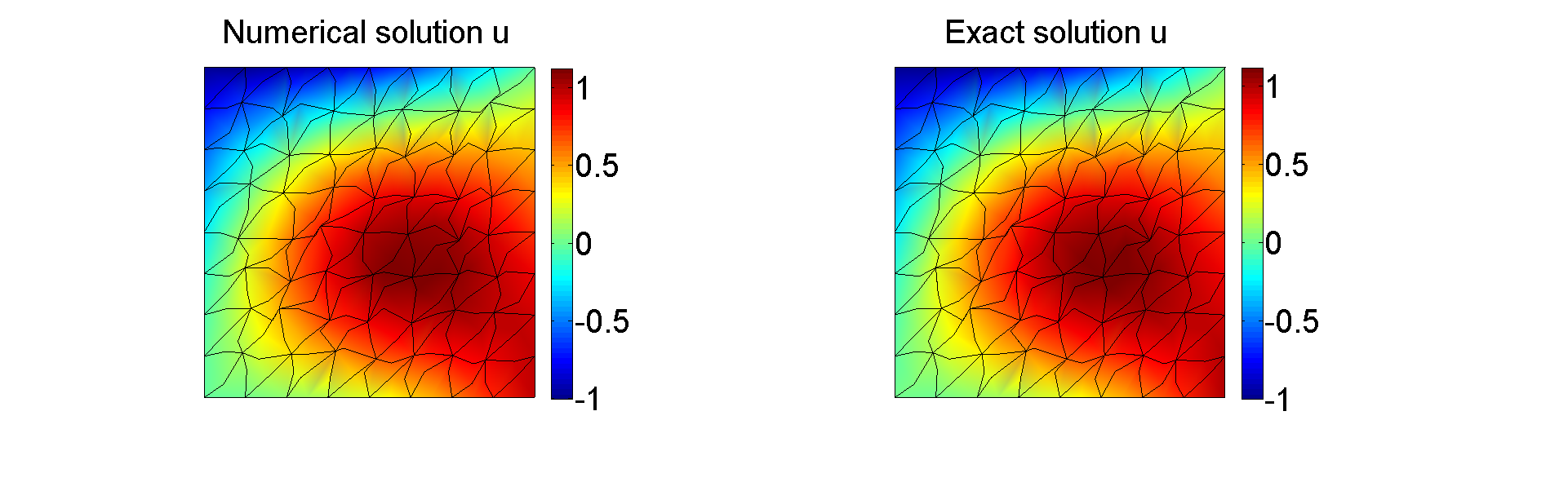}
\includegraphics [width=1.0\textwidth]{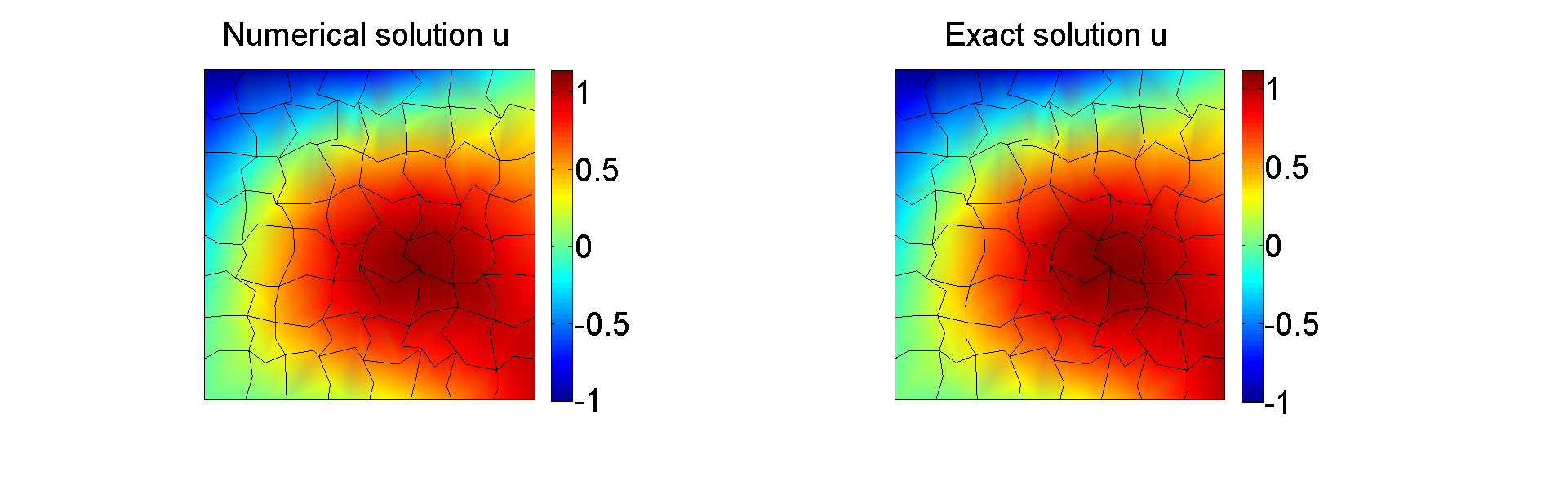}
\caption{\label{fig.comp_solution} Comparison of numerical solutions obtained from SWG and the exact solution for the test problem \eqref{eq.testcase.3} on various polygonal partitions of $h=1/8$ and $\kappa=1$.}
\end{figure}

\subsection{Numerical results on a non-convex domain}
The SWG scheme with the stabilization parameter $\kappa=4$ was applied to the test problem \eqref{eq.testcase.3} on the L-shaped domain $\Omega:=(-1,1)\times(-1,1)/ (0,1)\times(-1,0)$ partitioned into triangles or rectangles. The corresponding numerical results are summarized in Table \ref{tab.testcase5}, which shows a convergence of order $\O(h^2)$ in the $L^2$ norm for both the triangular and rectangular partitions. A superconvergence of order $\O(h^2)$ was observed in the discrete $H^1$ norm on rectangular partitions, while the optimal order of convergence with $r=1$ is confirmed numerically on triangular partitions. It should be pointed out that the $H^2$-regularity assumption \eqref{dual-regularity} is not valid for non-convex polygonal domains so that the optimal order of error estimate \eqref{eq.L2} is not known theoretically on the L-shaped domain. The numerical results therefore outperform the theory in the usual $L^2$ norm.

\begin{table}[h]
{\small
\begin{center}
\caption{Error and convergence performance of the SWG scheme \eqref{equation.SWG} for test case \eqref{eq.testcase.3} on Lshape domain, $\kappa = 4$.}\label{tab.testcase5}
\begin{tabular}{||c|cc|cc|cc|cc||}
\hline
     & \multicolumn{4}{| c |}{Triangular mesh} & \multicolumn{4}{| c |}{Square mesh}\\
\hline
$h^{-1}$ & $\|u_h-u \|_{0}$ & Rate & $\|u_h-u\|_{1}$ & Rate & $\|u_h-u \|_{0}$ & Rate & $\|u_h-u\|_{1}$ & Rate\\
\hline
 8  &   2.01e-02 &     -  &   4.31e-02  &    -  &   1.42e-02 &     -  &   2.55e-01  &    -\\
 16 &   5.02e-03 &  2.00  &   1.08e-02  & 2.00  &   3.56e-03 &  2.00  &   1.27e-01  & 1.01\\
 32 &   1.25e-03 &  2.00  &   2.70e-03  & 2.00  &   8.90e-04 &  2.00  &   6.35e-02  & 1.00\\
 64 &   3.14e-04 &  2.00  &   6.76e-04  & 2.00  &   2.23e-04 &  2.00  &   3.18e-02  & 1.00\\
 128&   7.84e-05 &  2.00  &   1.69e-04  & 2.00  &   5.57e-05 &  2.00  &   1.59e-02  & 1.00\\
\hline
\end{tabular}
\end{center}
}
\end{table}

\begin{figure}[h]
\includegraphics [width=1.0\textwidth]{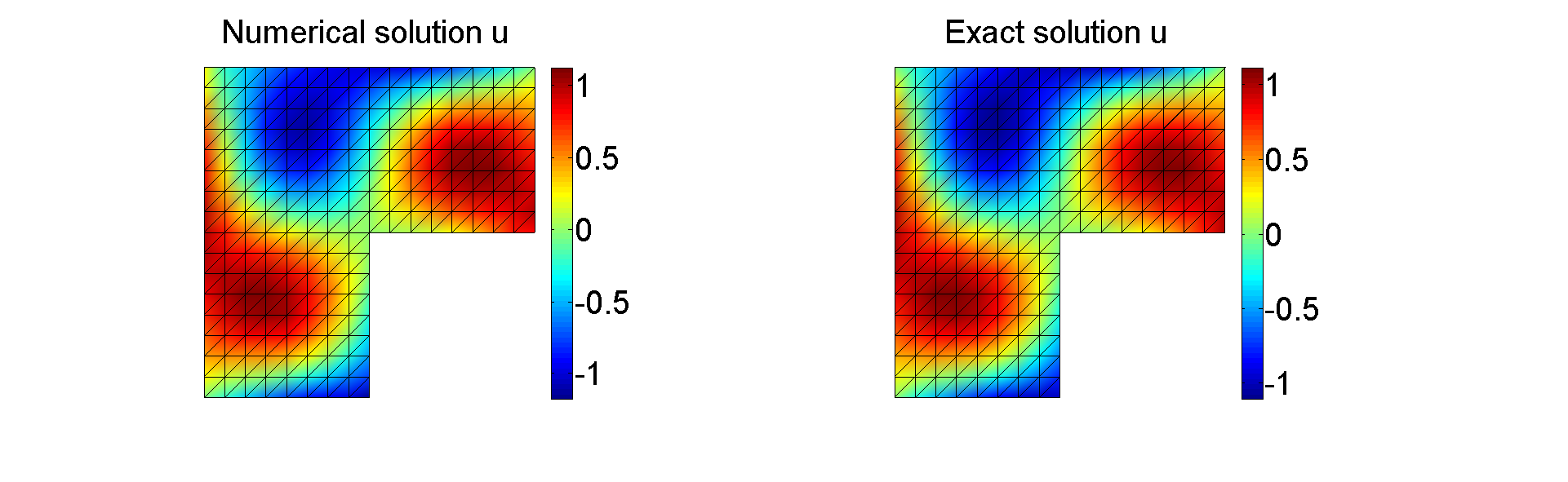}
\includegraphics [width=1.0\textwidth]{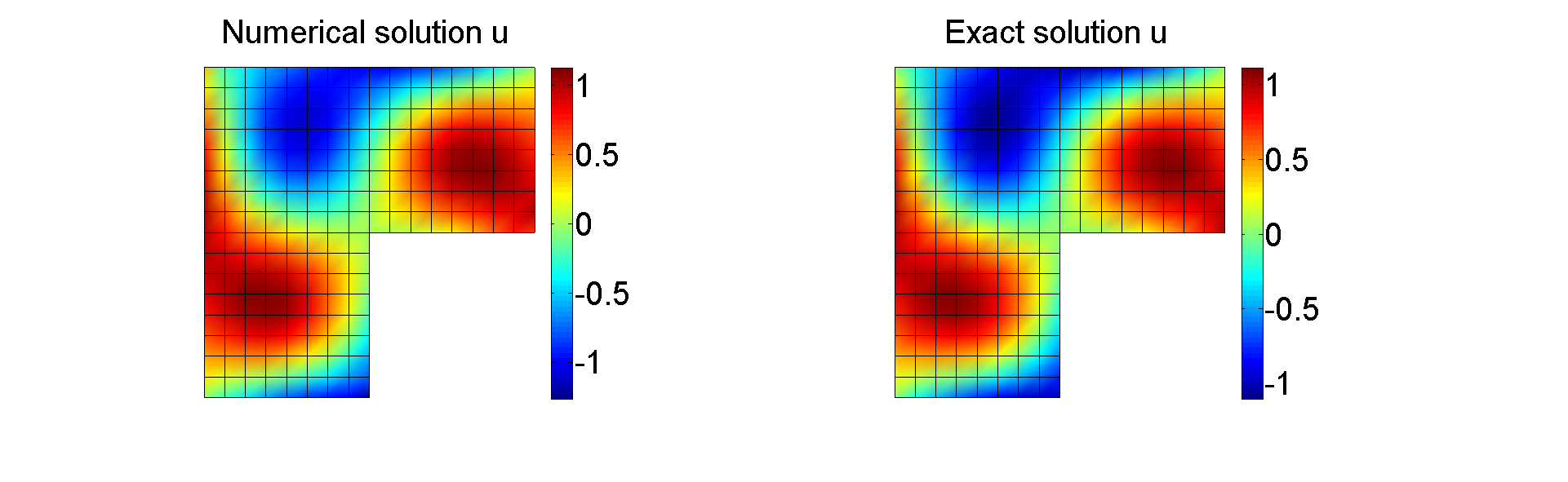}
\caption{\label{fig.comp_solution_Lshape} Comparison of numerical solution obtained from SWG and the exact solution for test case \eqref{eq.testcase.3} on L-shaped domain with $h=1/8$.}
\end{figure}

%It should be noted that the observed superconvergence in $H_1$ norm for the classic weak Galerkin finite element scheme has been proved in \cite{LiDanWW}. In summary, all the numerical results are in good consistency with the theory.

\newpage\newpage

\vfill\eject
%\newpage


\begin{thebibliography}{99}

%\bibitem{BertolazziManzini_2005}
%E. Bertolazzi, G. Manzini.
%\newblock{ A second-order maximum principle preserving finite volume method for steady convection¨Cdiffusion problems},
%{\it SIAM J. Numer. Anal.}, 43, pp. 2172-2199, 2005.

%\bibitem{ChristliebLiuTangXu_2015}
%A.J. Christlieb, Y. Liu, Q. Tang, Z. Xu.
%\newblock{ High order parametrized maximum-principle-preserving and positivity-preserving WENO schemes on unstructured meshes},
%{\it J. Comput. Phys.}, 281, pp. 334-351, 2015.

\bibitem{arnold-dg}
{\sc D. N. Arnold, F. Brezzi, B. Cockburn, and L. D. Marini},
{\em Uni?ed analysis of discontinuous Galerkin methods for elliptic problems}, SIAM J. Numer. Anal., 39 (2002), pp. 1749?779.

\bibitem{babuska}
{\sc I. Babu\u{s}ka}, {\em The finite element method with Lagrange
multipliers}, Numer. Math., 20 (1973), pp.~179-192.

\bibitem{babuska-dg}
{\sc I. Babu\u{s}ka}, {\em The ?nite element method with penalty}, Math. Comp., 27 (1973), pp. 221?28.

\bibitem{veiga-lipnikov-manzini}
{\sc L. Beiro da Veiga, K. Lipnikov, and G. Manzini}, {\em
Arbitrary-order nodal mimetic discretizations of elliptic problems
on polygonal meshes}, SIAM J. Numer. Anal. 49 (2011), 1737-1760.

\bibitem{veiga-lipnikov-manzini-2}
{\sc L. Beirao da Veiga, K. Lipnikov, and G. Manzini}, {\em
Convergence analysis of the high-order mimetic finite difference
method}, Numer. Math (2009) 113:325--356, DOI
10.1007/s00211-009-0234-6.

\bibitem{berndt}
{\sc M. Berndt, K. Lipnikov, J. D. Moulton, and M. Shashkov}, {\em
Convergence of mimetic finite difference discretizations of the
diffusion equation}, East-West J. Numer. Math. 9 (2001),
pp.~253-294.

%\bibitem{CiarletRaviart_1973}
%P. Ciarlet, P.-A. Raviart.
%\newblock{ Maximum principle and uniform convergence for the finite element method},
%{\it Comput. Methods Appl. Mech. Eng.}, 2, pp. 17-31, 1973.

%\bibitem{DroniouPotier_2011}
%J. Droniou, C.L. Potier.
%\newblock{ Construction and convergence study of schemes preserving the elliptic local maximum principle},
%{\it SIAM J. Numer. Anal.}, 49, pp. 459-490, 2011.

%\bibitem{DraganescuDupontScott_2004}
%A. Dr\u{a}g\u{a}nescu, T.F. Dupont, L.R. Scott.
%\newblock{ Failure of the discrete maximum principle for an elliptic finite %{\it Math. Comput.}, 74, pp. 1-24, 2004.

%\bibitem{HuangWang_2015}
%W. Huang, Y. Wang.
%\newblock{ Discrete maximum principle for the weak Galerkin method for anisotropic diffusion problems},
%{\it Commun. Comput. Phys.}, 18, pp. 65-90, 2015.

\bibitem{bdm}
{\sc F. Brezzi, J. Douglas, Jr., and L.D. Marini}, {\em Two families
of mixed finite elements for second order elliptic problems}, Numer.
Math., 47 (1985), pp.~217-235.

\bibitem{sue}
{\sc S. Brenner and R. Scott}, {\em The Mathematical Theory of
Finite Element Mathods},  Springer-Verlag, New York, 1994.

\bibitem{Chen-CD-WG}
 {\sc G. Chen, M. Feng, and X. Xie},
 {\em A robust WG finite element method for convection-diffusion-reaction equations}, J. Comput. Appl. Math., 315 (2017), pp. 107?25.

\bibitem{ciarlet-fem}
{\sc P.G. Ciarlet},
\textit{The Finite Element Method for Elliptic
Problems}, Classics Appl. Math. 40, SIAM, Philadelphia, 2002.

\bibitem{cgl}
{\sc B. Cockburn, J. Gopalakrishnan, and R. Lazarov}, {\em Unified
hybridization of discontinuous Galerkin, mixed, and continuous
Galerkin methods for second order elliptic problems}, SIAM J. Numer.
Anal. 47 (2009), pp.~1319-1365.

\bibitem{cs-ldg}
{\sc B. Cockburn and C.-W. Shu},
{\em The local discontinuous Galerkin method for time-dependent convection-diffusion systems},  SIAM Journal on Numerical Analysis,  35 (1998),  2440-2463.

\bibitem{pietro-ern}
{\sc D.A. Di Pietro and A. Ern},
{\em Mathematical Aspects of
Discontinuous Galerkin Methods}, Springer-Verlag Berlin Heidelberg,
2012.

\bibitem{douglas-dg}
{\sc J. Douglas Jr. and T. Dupont},
{\em Interior penalty procedures for elliptic and parabolic Galerkin methods}, in Second International Symposium on Computing Methods in Applied Sciences, Versailles, 1975. Lecture Notes in Phys. 58, Springer, Berlin, 1976, pp. 207?16.

\bibitem{fv1965}
{\sc B. Fraeijs de Veubeke},
{\em Displacement and equilibrium models in the finite element method}. In: {\em Stress Analysis}, O. C. Zienkiewicz and G. Holister (eds.).
New York: John Wiley, 1965.

\bibitem{gr}
{\sc V. Girault and P. A. Raviart}, {\em Finite Element Methods
for the Navier-Stokes Equations: Theory and Algorithms},
Springer-Verlag, Berlin, 1986.

\bibitem{hesthaven}
{\sc J. S. Hesthaven and T. Warburton},
{\em Nodal Discontinuous Galerkin Methods: Algorithms, Analysis, and Applications}, Texts Appl. Math. 54, Springer, New York, 2008.

\bibitem{LiDanWW}
{\sc D. Li, C. Wang, and J. Wang},
\newblock{\em Superconvergence of the gradient approximation for weak Galerkin finite element methods on nonuniform rectangular partitions}, https://arxiv.org/pdf/1804.03998v2.pdf.

\bibitem{LiWang_2013}
{\sc Q. Li and J. Wang},
\newblock{\em Weak Galerkin finite element methods for parabolic equations},
{Numer. Methods Partial Differ. Equ.}, 29, pp. 1-21, 2013.

\bibitem{LinRunchang}
{\sc R. Lin, X. Ye, S. Zhang, and P. Zhu},
\newblock{\em A weak Galerkin finite element method for singularly
perturbed convection-diffusion-reaction problems},
{\it SIAM Journal on Numerical Analysis}, 2018, Vol. 56, No. 3 : pp. 1482-1497.

\bibitem{LiuWang_SWG_Stokes_2018}
{\sc Y. Liu and J. Wang},
\newblock{\em Simplified weak Galerkin and finite difference schemes for the Stokes equation}, arXiv:1803.00120, 2018.

\bibitem{MWY-HWG}
{\sc L. Mu, J. Wang, and X. Ye},
{\em A hybridized formulation for the weak Galerkin mixed finite element method}, Journal of Computational and Applied Mathematics, Volume 307, 2016, pp.
335-345. doi:10.1016/j.cam.2016.01.004.

\bibitem{MWWeiYZhao_2013}
{\sc L. Mu, J. Wang, G. Wei, X. Ye, and S. Zhao},
\newblock{\em Weak Galerkin methods for second order elliptic interface problems},
{J. Comput. Phys.}, 250, pp. 106-125, 2013.

\bibitem{mwy}
 {\sc L. Mu, J. Wang and X. Ye}, {\em
 A weak Galerkin finite element method with polynomial reduction},
 Journal of Computational and Applied Mathematics, vol. 285, pp. 45-58, 2015.

%\bibitem{Mudunuro}
%M. K. Mudunuru and K. B. Nakshatrala.
%\newblock On enforcing maximum principles and achieving element-wise species balance for advection-diffusion-reaction equations under the finite element method.
%{\em Journal of Computational Physics}, Volume 305, 15, pp. 448¨C493, 2016.
\bibitem{nitsche-dg}
 {\sc J. Nitsche},
 {\em \"{U}ber ein Variationsprinzip zur L\"{o}osung von Dirichlet-Problemen bei Verwendung von Teilra\"{u}men}, die keinen Randbedingungen unterworfen sind, Abh. Math. Sem. Univ. Hamburg, 36 (1971), pp. 9?5.

\bibitem{rt}
{\sc P. Raviart and J. Thomas}, {\em A mixed finite element method
for second order elliptic problems}, Mathematical Aspects of the
Finite Element Method, I. Galligani, E. Magenes, eds., Lectures
Notes in Math. 606, Springer-Verlag, New York, 1977.

\bibitem{riviere}
{\sc B. Rivi\`ere},
{\em Discontinuous Galerkin Methods for Solving Elliptic and Parabolic Equations. Theory and Implementation}, Front. Appl. Math. 35, SIAM, Philadelphia, 2008.

\bibitem{StrangFix}
{\sc G. Strang and G. J. Fix},
\newblock{\em An Analysis of the Finite Element Method},
{Prentice Hall, Englewood Cliffs, NJ}, 1973.

%\bibitem{Varga_1966}
%R.S. Varga.
%\newblock{ On a discrete maximum principle},
%{\it SIAM J. Numer. Anal.}, 3, pp. 355-359, 1966.

\bibitem{ww-hwg}
{\sc C. Wang and J. Wang},
{\em A hybridized weak Galerkin finite element method for the biharmonic equation}, arXiv:1402.1157, International Journal of Numerical Analysis and Modeling, Volume 12, Number 2, pp. 302-317, 2015.


\bibitem{WangWang_2016} {\sc C. Wang and J. Wang},
{\em A primal-dual weak Galerkin finite element method for second order elliptic equations in non-divergence form}, Math. Comp., vol. 87, 515-545, 2018. DOI: https://doi.org/10.1090/mcom/3220. June 2017.

\bibitem{ww2017} {\sc  C. Wang and J. Wang}, {\em A Primal-Dual weak Galerkin finite element method for Fokker-Planck type equations}, arXiv:1704.05606, SIAM Journal of Numerical Analysis, accepted.

\bibitem{WangWang-EC} {\sc C. Wang and J. Wang},
{\em Primal-Dual weak Galerkin finite element methods for elliptic Cauchy problems}, arXiv:1806.01583 [math.NA], submitted for publication.

\bibitem{LockingFree}
{\sc C. Wang, J. Wang, R. Wang, and R. Zhang},
{\em A locking-free weak Galerkin finite element method for elasticity
problems in the primal formulation}, Journal of Computational and
Applied Mathematics, doi:10.1016/j.cam.2015.12.015, Vol 307, 2016,
pp. 346-366.

\bibitem{ww-survey}
{\sc J. Wang and C. Wang},
{\em Weak Galerkin finite element methods
for elliptic PDEs (in Chinese)},  Sci. Sin. Math.,  45 (2015),  1061-1092,
doi:10.1360/N012014-00233.


\bibitem{wg-systematic} {\sc J. Wang, R. Wang, Q. Zhai, and R. Zhang},
{\em A systematic study on weak Galerkin finite element methods for second order elliptic problems}, J. Sci. Comput. (2018) 74: 1369. https://doi.org/10.1007/s10915-017-0496-6.

\bibitem{WangYe_2013}
{\sc J. Wang and X. Ye},
\newblock{\em A weak Galerkin mixed finite element method for second-order ellliptic problems}, {J. Comp. and Appl. Math.}, { 241}, 103-115, 2013.

\bibitem{wy3655}
{\sc J. Wang and X. Ye},
\newblock{\em A weak Galerkin mixed finite element method for second-order elliptic problems}, {Math. Comp.}, 83, pp. 2101-2126, 2014.

%\bibitem{WYZhaiZhang_2018}
%J. Wang , X. Ye , Q. Zhai, R. Zhang.
%\newblock{Discrete maximum principle for the $P_1$-$P_0$ weak Galerkin finite element approximations}, {\it J. Comput. Phys.}, 362, pp. 114-130, 2018.

%\bibitem{WangZhang_2012}
%J. Wang, R. Zhang.
%\newblock{ Maximum principles for P1-conforming finite element approximations of quasi-linear second order elliptic equations},
%{\it SIAM J. Numer. Anal.}, 50, pp. 626-642, 2012.

%\bibitem{ZhZhShu_2013}
%Y. Zhang, X. Zhang, C.-W. Shu.
%\newblock{ Maximum-principle-satisfying second order discontinuous Galerkin schemes for convection¨Cdiffusion equations on triangular meshes},
%{\it J. Comput. Phys.}, 234, pp. 295-316, 2013.

\bibitem{wheeler-dg}
{\sc M. F. Wheeler},
{\em An elliptic collocation-finite element method with interior penalties}, SIAM J. Numer. Anal., 15 (1978), pp. 152?61.

\bibitem{ZhangTie}
{\sc T. Zhang and Y. Chen},
\newblock{\em An analysis of the weak finite element method for convection-diffusion equations}, {arXiv:1506.02793 [math.NA]}, June 2015.

\end{thebibliography}
\end{document}